\documentclass[twoside, 11pt]{article}
\usepackage{amsmath,amssymb,amsthm,mathrsfs}
\usepackage[margin=2.5cm]{geometry} % showframe,
\usepackage{lipsum}
\usepackage{titlesec,hyperref}
\usepackage{fancyhdr}
\usepackage[numbers,sort&compress]{natbib}
\usepackage{color}

\fancypagestyle{plain}
{
\fancyhf{}

}

\titleformat{\subsection}{\it}{\thesubsection.\enspace}{1.5pt}{}
\titleformat{\subsubsection}{\it}{\thesubsubsection.\enspace}{1.5pt}{}

\newtheorem{theo}{Theorem}[section]
\newtheorem{lemm}[theo]{Lemma}

\newtheorem{prop}[theo]{Proposition}
\newtheorem{rema}{Remark}[section]
\numberwithin{equation}{section}

\allowdisplaybreaks

\def\th2{\frac{\theta}{2}}

\begin{document}
\title{Global Existence and Optimal Decay Rates of Solutions for Compressible Hall-MHD Equations \hspace{-4mm}}
\author{Jincheng Gao \quad Zheng-an Yao \\[10pt]
\small {School of Mathematics and Computational Science, Sun Yat-sen University,}\\
\small {510275, Guangzhou, PR China}\\[5pt]
}

\footnotetext{Email: \it gaojc1998@163.com(J.C.Gao), \it mcsyao@mail.sysu.edu.cn(Z.A.Yao).}
\date{}

\maketitle

\begin{abstract}
In this paper, we are concerned with global existence and optimal decay rates of
solutions for the three-dimensional compressible Hall-MHD equations.
First, we prove the global existence of
strong solutions by the standard energy method under the condition
that the initial data are close to the constant equilibrium state in $H^2$-framework.
Second, optimal decay rates of strong solutions in $L^2$-norm are obtained
if the initial data belong to $L^1$ additionally.
Finally, we apply Fourier splitting method by
Schonbek [Arch.Rational Mech. Anal. 88 (1985)] to establish optimal decay rates for
higher order spatial derivatives of classical solutions in $H^3$-framework,
which improves the work of Fan et al.[Nonlinear Anal. Real World Appl. 22 (2015)].

\vspace*{5pt}
\noindent{\it {\rm Keywords}}:
compressible Hall-MHD equations;
global solutions; optimal decay rates; Fourier splitting method.

\vspace*{5pt}
\noindent{\it {\rm 2010 Mathematics Subject Classification}}:
76W05, 35Q35, 35D05, 76X05.
\end{abstract}

%\tableofcontents

\section{Introduction}
\quad The application of Hall-magnetohydrodynamics system (in short, Hall-MHD)
covers a very wide range of physical objects, for example, magnetic reconnection
in space plasmas, star formulation, neutron stars, and geodynamo, refer to
\cite{{Forbes},{Homann},{Wardle},{Balbus-Terquem},{Shalybkov-Urpin},{Mininni}}
and the references therein.
Recently, Acheritogaray et al.\cite{Liu} derived the Hall-MHD equations
from the two-fluid Euler-Maxwell system for
electrons and ions through a set of scaling limits or from the kinetic
equations by taking macroscopic quantities in the equations under some closure assumptions.
They also established the global existence of weak solutions for periodic boundary condition.
In this paper, we investigate the following compressible Hall-MHD equations in
three dimensional whole space $\mathbb{R}^3$(see \cite{Liu}):
\begin{equation}\label{1.1}
\left\{
\begin{aligned}
&\rho_t+{\rm div}(\rho u)=0,\\
&(\rho u)_t+{\rm div}(\rho u\otimes u)-\mu \Delta u-(\mu+\nu)\nabla {\rm div}u
  +\nabla P(\rho)=({\rm curl} B)\times B,\\
&B_t -{\rm curl}(u \times B)+{{\rm curl} { \left[\frac{({\rm curl} B)\times B}{\rho}\right]}}= \Delta B, \ {\rm div} B=0,
\end{aligned}
\right.
\end{equation}
where the functions $\rho, u,$ and $B$
represent density, velocity, and magnetic field respectively.
The pressure $P(\rho)$ is a smooth function in a neighborhood of $1$ with $P'(1)=1$.
The constants $\mu$ and $\nu$ denote the viscosity coefficients of the flow
and satisfy physical condition as follows
$$
\mu>0, \ 2\mu+3\nu \ge 0.
$$
To complete the system \eqref{1.1}, the initial data are given by
\begin{equation}\label{1.2}
\left.(\rho, u, B)(x,t)\right|_{t=0}=(\rho_0(x), u_0(x), B_0(x)).
\end{equation}\label{1.3}
Furthermore, as the space variable tends to infinity, we assume
\begin{equation}\label{1.3}
\underset{|x|\rightarrow \infty}{\lim}(\rho_0-1, u_0, B_0)(x)=0.
\end{equation}
Obviously, the compressible Hall-MHD equations transform into the well-known
compressible MHD equations when the Hall effect term
${\rm curl}\left(\frac{({\rm curl}B)\times B}{\rho}\right)$ is neglected.

When the density is constant,
Chae et al.\cite{Chae-Degond-Liu} proved local existence of smooth solutions for large data
and global smooth solutions for small data in three dimensional whole space.
They also showed a Liouville theorem for the stationary solutions.
Chae and Lee \cite{Chae-Lee} established an optimal blow-up criterion for classical solutions
and proved two global-in-time existence results of classical solutions for small initial data,
the smallness conditions of which are given by the suitable Sobolev and Besov norms respectively.
Later, Fan et al.\cite{Fan-Li-Nakamura} also established some new regularity criteria,
which also are built for density-dependent incompressible Hall-MHD equations with
positive initial density by Fan and Ozawa \cite{Fan-Ozawa}.
Maicon and Lucas \cite{Maicon-Lucas} proved a stability theorem for global large solutions under a suitable integrable hypothesis and constructed a special large solution by assuming the condition of curl-free magnetic fields.
Fan et al. \cite{Fan-Huang-Nakamura} established the global well-posedness of the axisymmetric solutions.
Chae and Schonbek \cite{Chae-Schonbek} established temporal decay estimates for weak solutions
and obtained algebraic time decay for higher order Sobolev norms of small initial data solutions
as follows
$$
\|\nabla^k u(t)\|_{L^2}+\|\nabla^k B(t)\|_{L^2}\le C(1+t)^{-\frac{3+2k}{4}},
~k \in \mathbb{N}
$$
for all $t \ge T^*$($T^*$ is a positive constant).
Furthermore, Weng \cite{Weng-Shangkun} extended this result
by providing upper and lower bounds on the decay of higher order derivatives.
For the compressible Hall-MHD equations \eqref{1.1},
Fan et al.\cite{Fan-Zhou} proved the local existence of strong solutions with
positive initial density and global small solutions(classical solutions) with small initial perturbation.
They also established optimal time decay rate for classical solutions as follows
\begin{equation}\label{1.4}
\|(\rho-1, u, B)(t)\|_{L^2} \le C(1+t)^{-\frac{3}{4}}.
\end{equation}
Here, they required the initial perturbation is small in $H^3$-norm
and bounded in $L^1$-norm.

Recently, the study of decay rates for solutions to the MHD equations has
aroused many researchers' interest.
First of all, under the $H^3$-framework,  Li and Yu \cite{Li-Yu} and Chen and Tan \cite{Chen-Tan}
not only established the global existence of classical solutions, but also obtained
the time decay rates for the three-dimensional compressible MHD equations by assuming the initial data belong to $L^1$ and $L^q( q \in \left[1, \frac{6}{5}\right))$ respectively.
More precisely, Chen and Tan \cite{Chen-Tan} built the time decay rates
\begin{equation}\label{1.5}
\|\nabla^k (\rho-1, u, B)(t)\|_{H^{3-k}}\le C(1+t)^{-\frac{3}{2}\left(\frac{1}{q}-\frac{1}{2}\right)-\frac{k}{2}},
\end{equation}
where $k=0,1$. The time decay rates \eqref{1.5} has also been established by
Li and Yu \cite{Li-Yu} for the case $q=1$.
Motivated by the work of Guo and Wang \cite{Guo-Wang},
Tan and Wang \cite{Tan-Wang} established
the optimal time decay rates for the higher order spatial derivatives of solutions
if the initial perturbation belongs to $H^N\cap\dot{H}^{-s}\left(N\ge3, s \in \left[0, \frac{3}{2}\right)\right)$.
More precisely, they built the following time decay rates
$$
\|\nabla^k (\rho-1, u, B)(t)\|_{H^{N-k}}\le C(1+t)^{-\frac{k+s}{2}},
$$
where $k=0,1,...,N-1$.
Motivated by the work \cite{Gao-Tao-Yao}, we (see \cite{Gao-Chen-Yao})
establish the following time decay rates
for all $t \ge T^*(T^*$ is a positive constant$)$,
\begin{equation}\label{1.6}
\begin{aligned}
&\|\nabla^k(\rho-1)(t)\|_{H^{3-k}}+\|\nabla^k u(t)\|_{H^{3-k}}\le C(1+t)^{-\frac{3+2k}{4}},\\
&\|\nabla^m B(t)\|_{H^{3-m}}\le C(1+t)^{-\frac{3+2m}{4}},\\
\end{aligned}
\end{equation}
where $k=0, 1, 2,$ and $m=0,1,2,3$.  It is easy to see that the time decay rates \eqref{1.6}
is better than decay rates \eqref{1.5} since \eqref{1.6} provides faster time decay rates
for the higher order spatial derivatives of solutions.

In this paper, we hope to establish the global existence and time decay rates
of solutions for the compressible Hall-MHD equations \eqref{1.1}-\eqref{1.3}.
First of all, we construct the global existence of strong solutions by the standard
energy method under the condition that the initial data are close to the constant equilibrium
state $(1, 0, 0)$ in $H^2$-norm.
Second, if the initial data in $L^1-$norm are finite additionally, the optimal
time decay rates of strong solutions are established by the method of Green function.
Precisely, we obtain the following time decay rates for all $ t\ge 0$,
$$
\|(\rho-1)(t)\|_{H^{2-k}}+\|u(t)\|_{H^{2-k}}+\|B(t)\|_{H^{2-k}}\le C(1+t)^{-\frac{3+2k}{4}},
$$
where $k=0, 1$.
This framework of time convergence rates for compressible flows has been applied to
other compressible models, refer to \cite{{Tan-Wang2},{Hu-Wu2},{Wang-Wang}, {Wang-Wen-Jun}}.
Although magnetic field equations \eqref{1.1}$_3$ are nonlinear parabolic equations,
we hope to establish optimal time decay rates for the second order spatial derivatives
of magnetic field under the condition of small initial perturbation.
In order to achieve this goal, we move the nonlinear terms
to the right hand side of \eqref{1.1}$_3$ and deal with the nonlinear terms as external
force with the property on fast time decay rates.
Then, the application of Fourier splitting method by Schonbek \cite{Schonbek} helps
us to establish optimal time decay rate for the second order spatial derivatives
of magnetic field as follows
$$
\|\nabla^2 B(t)\|_{L^2}\le C(1+t)^{-\frac{7}{4}}.
$$
Finally, one focus on establishing optimal time decay rates for higher order spatial
derivatives of classical solutions to compressible Hall-MHD equations.
More precisely, we prove that the global classical solution $(\rho, u, B)$ of Cauchy problem
\eqref{1.1}-\eqref{1.3} has the time decay rates \eqref{1.6}.
Obviously, these time decay rates improve the results \eqref{1.4}
by Fan et al.\cite{Fan-Zhou} since we  build faster time decay rates for higher order
spatial derivatives of classical solutions.

\textbf{Notation:} In this paper, we use $H^s(\mathbb{R}^3)( s\in \mathbb{R})$ to
denote the usual Sobolev spaces
with norm $\|\cdot\|_{H^s}$ and $L^p(\mathbb{R}^3)(1\le p \le \infty)$ to denote the usual $L^p$ spaces with norm $\| \cdot \|_{L^p}$. The symbol $\nabla^l $ with an integer $l \ge 0$ stands for the usual any spatial derivatives of order $l$.
For example,  we define
$$
\nabla^k v=
\left\{\left.\partial_x^\alpha v_i\right||\alpha|=k,~i=1,2,3\right\}
,~v=(v_1, v_2, v_3).
$$
We also denote $\mathscr{F}(f):=\hat{f}$.
The notation $a \lesssim b$ means that $a \le C b$ for a universal constant $C>0$ independent of
time $t$.
The notation $a \approx b$ means $a \lesssim b$ and $b \lesssim a$.
For the sake of simplicity, we write $\int f dx:=\int _{\mathbb{R}^3} f dx.$

First of all, we establish the global existence and optimal decay rates of strong solutions
for the compressible Hall-MHD equations \eqref{1.1}-\eqref{1.3}.
\begin{theo}\label{THM1}
Assume that the initial data $(\rho_0-1,u_0, B_0)\in H^2$ and there exists a small
constant $\delta_0>0$ such that
$$
\|(\rho_0-1, u_0, B_0)\|_{H^2} \le \delta_0,
$$
then the problem \eqref{1.1}-\eqref{1.3} admits a unique global strong solution $(\rho, u, B)$
satisfying for all $t \ge 0$,
$$
\|(\rho-1,u,B)(t)\|_{H^2}^2+\int_0^t (\|\nabla \rho(s)\|_{H^1}^2+\|\nabla(u, B)(s)\|_{H^2}^2)ds
\le C\|(\rho_0-1,u_0,B_0)\|_{H^2}^2.
$$
Furthermore, if $\|(\rho_0-1, u_0, B_0)\|_{L^1}$ is finite additionally,  then the global strong
solution $(\rho,u, B)$ has following decay rates for all $t \ge 0$,
\begin{equation}\label{Decay1}
\begin{aligned}
&\|\nabla^k(\rho-1)(t)\|_{H^{2-k}}+\|\nabla^k u(t)\|_{H^{2-k}}\le C(1+t)^{-\frac{3+2k}{4}},\\
&\|\nabla^m B(t)\|_{H^{2-m}}\le C(1+t)^{-\frac{3+2m}{4}},\\
\end{aligned}
\end{equation}
where $k=0,1$, and $m=0,1,2$.
\end{theo}

\begin{rema}
For any $2\le p \le6$, by virtue of Theorem \ref{THM1} and the Sobolev interpolation inequality,
we obtain time decay rates as follows
$$
\begin{aligned}
&\|(\rho-1)(t)\|_{L^p}+\|u(t)\|_{L^p}\le C(1+t)^{-\frac{3}{2}\left(1-\frac{1}{p}\right)},\\
&\|\nabla^k B(t)\|_{L^p}\le C(1+t)^{-\frac{3}{2}\left(1-\frac{1}{p}\right)-\frac{k}{2}},\\
\end{aligned}
$$
where $k=0,1.$ Furthermore, in the same manner, we also have
$$
\begin{aligned}
&\|(\rho-1)(t)\|_{L^\infty}+\|u(t)\|_{L^\infty}\le C(1+t)^{-\frac{5}{4}},\\
&\|B (t)\|_{L^\infty}\le C(1+t)^{-\frac{3}{2}}.
\end{aligned}
$$
\end{rema}

Second, we build time decay rates for the time derivatives of global strong solutions.

\begin{theo}\label{THM2}
Under all the assumptions in Theorem \ref{THM1}, the global strong solution $(\rho, u, B)$
of Cauchy problem \eqref{1.1}-\eqref{1.3} has the decay rates
$$
\begin{aligned}
&\|\rho_t(t)\|_{H^1}+\|u_t(t)\|_{L^2}\le C(1+t)^{-\frac{5}{4}},\\
&\|B_t(t)\|_{L^2}\le C(1+t)^{-\frac{7}{4}}
\end{aligned}
$$
for all $t \ge 0$.
\end{theo}

Furthermore, we establish optimal decay rates for the higher order spatial derivatives of
classical solutions to the compressible Hall-MHD equations.

\begin{theo}\label{THM3}
Assume that the initial data $(\rho_0-1,u_0,B_0)\in H^3\cap L^1$ and there exists
a small constant $\varepsilon_0>0$ such that
\begin{equation}\label{smallness}
\|(\rho_0-1,u_0, B_0)\|_{H^3} \le \varepsilon_0,
\end{equation}
then the global classical solution $(\rho, u, B)$ of the problem \eqref{1.1}-\eqref{1.3}
has the time decay rates
\begin{equation}\label{Decay3}
\begin{aligned}
&\|\nabla^k(\rho-1)(t)\|_{H^{3-k}}+\|\nabla^k u(t)\|_{H^{3-k}}\le C(1+t)^{-\frac{3+2k}{4}},\\
&\|\nabla^m B(t)\|_{H^{3-m}}\le C(1+t)^{-\frac{3+2m}{4}},\\
\end{aligned}
\end{equation}
where $k=0, 1, 2,$ and $m=0,1,2,3.$
\end{theo}

\begin{rema}
Compared with the decay rates of linearized systems of
\eqref{1.1} stated in Proposition \ref{linearized-Decay},
\eqref{Decay3} gives optimal decay rates
of the solutions and its spatial derivatives (except for the
third order spatial derivatives of density and velocity)
in $L^2$-norm to the nonlinear problem \eqref{1.1}-\eqref{1.3}.
Here the decay rate of solutions to nonlinear system is optimal
in the sense that it coincides with the rate of solutions to the linearized systems.
\end{rema}

\begin{rema}
By virtue of the Sobolev inequality and the results \eqref{Decay3} in Theorem \ref{THM3},
then the global classical solution $(\rho, u, B)$ has the time decay rates
$$
\begin{aligned}
&\|(\rho-1)(t)\|_{L^p}+\|u (t)\|_{L^p} \le C(1+t)^{-\frac{3}{2}\left(1-\frac{1}{p}\right)},\\
&\|\nabla^k B(t)\|_{L^p} \le C(1+t)^{-\frac{3}{2}\left(1-\frac{1}{p}\right)-\frac{k}{2}},
\end{aligned}
$$
where $k=0,1$, and $p \in [2, \infty]$.
Hence, the decay rate of classical solution $(\rho, u, B)$
converging to the equilibrium state
$(1, 0, 0)$ in $L^\infty$-norm is $(1+t)^{-\frac{3}{2}}$.
\end{rema}

\begin{rema}
It is easy to see that \eqref{Decay3} provides faster time decay rates
for higher order spatial derivatives of global classical solutions than \eqref{1.4}.
Hence, the results in Theorem \ref{THM3} improve the work of Fan et al. \cite{Fan-Zhou}.
\end{rema}

\begin{rema}
Although we only established the time decay rates under the $H^3$-framework
in Theorem \ref{THM3}, the method here can be
applied to the $H^N(N \ge 3)$-framework just following the idea
as Gao et al.\cite{Gao-Tao-Yao}. Hence,
if $(\rho_0-1, u_0, B_0)\in H^N \cap L^1(N \ge3)$,
then the global solution $(\rho, u, B)$ has the time decay rates
$$
\begin{aligned}
&\|\nabla^k(\rho-1)(t)\|_{H^{N-k}}+\|\nabla^k u(t)\|_{H^{N-k}}\le C(1+t)^{-\frac{3+2k}{4}},\\
&\|\nabla^m B(t)\|_{H^{N-m}}\le C(1+t)^{-\frac{3+2m}{4}},\\
\end{aligned}
$$
where $k=0, 1,..., N-1,$ and $m=0,1,2,..., N.$
\end{rema}

Finally, we build decay rates for the mixed space-time derivatives of global classical solutions.

\begin{theo}\label{THM4}
Under all the assumptions in Theorem \ref{THM3}, the global classical solution $(\rho, u, B)$
of the problem \eqref{1.1}-\eqref{1.3} satisfies the time decay rates
$$
\begin{aligned}
&\|\nabla^k \rho_t(t)\|_{H^{2-k}}+\|\nabla^k u_t(t)\|_{L^2}\le C(1+t)^{\frac{5+2k}{4}},\\
&\|\nabla^k B_t(t)\|_{L^2}\le C(1+t)^{-\frac{7+2k}{4}},
\end{aligned}
$$
where $k=0,1$.
\end{theo}

This paper is organized as follows.
In section $2$, we establish some energy estimates that will
play an essential role for us to construct the global existence of strong solutions.
Then, we close the estimates by the standard continuity argument
and the global existence of strong solutions follows immediately.
Furthermore, we build the time decay rates by taking the method of Green function
and establish optimal time decay rates for the second order spatial derivatives
of magnetic field.
Finally, we also study decay rates for the time derivatives of density, velocity
and magnetic field.
In section $3$, we establish the optimal decay rates for the higher order spatial
derivatives of global classical solutions and mixed space-time derivatives of
solutions.

\section{Proof of Theorem \ref{THM1} and Theorem \ref{THM2}}

\quad In this section, we will establish global existence and optimal time decay rates
of strong solutions for the compressible Hall-MHD equations.
Indeed, computing directly, it is easy to deduce
$$
({\rm curl} B) \times B=(B\cdot \nabla)B-\frac{1}{2}\nabla (|B|^2),
$$
and
$$
{\rm curl}{(u \times B)}=u({\rm div}B)-(u\cdot \nabla) B+(B\cdot \nabla)u-B({\rm div}u).
$$
Then, denoting $\varrho=\rho-1$, we rewrite \eqref{1.1} in the perturbation form as
\begin{equation}\label{eq1}
\left\{
\begin{aligned}
&\varrho_t+{\rm div}u=S_1,\\
&u_t-\mu \Delta u-(\mu+\nu)\nabla {\rm div}u+\nabla \varrho=S_2,\\
&B_t-\Delta B=S_3, \quad {\rm div}B=0,
\end{aligned}
\right.
\end{equation}
where the function $S_i(i=1,2,3)$ is defined as
\begin{equation}\label{eq2}
\left\{
\begin{aligned}
&S_1=-\varrho{\rm div}u-u\cdot \nabla \varrho,\\
&\!S_2\!=-u\!\cdot \!\nabla u-\!h(\varrho)[\mu\Delta u+(\mu+\nu)\nabla {\rm div}u]
 -f(\varrho)\nabla \varrho+g(\varrho)\!\left[B\cdot \nabla B-\frac{1}{2}\nabla(|B|^2)\!\right],\\
&S_3=-u\cdot \nabla B+B\cdot \nabla u-B{\rm div}u
-{\rm curl}\left[g(\varrho)\!\left(B\cdot \nabla B-\frac{1}{2}\nabla(|B|^2)\!\right)\right].
\end{aligned}
\right.
\end{equation}
Here the nonlinear function of $\varrho$ is defined by
\begin{equation}\label{eq3}
h(\varrho)=\frac{\varrho}{\varrho+1}, \quad
f(\varrho)=\frac{P'(\varrho+1)}{\varrho+1}-1, \quad
g(\varrho)=\frac{1}{\varrho+1}.
\end{equation}
The initial data are given as
\begin{equation}\label{eq4}
\left.(\varrho, u, B)(x,t)\right|_{t=0}=(\varrho_0, u_0, B_0)(x)
\rightarrow(0,0,0) \quad {\text {as} } \quad |x|\rightarrow \infty.
\end{equation}

\subsection{Energy estimates}
\quad First of all, suppose there exists a small positive constant $\delta$
satisfying following estimate
\begin{equation}\label{2.5}
\|(\varrho, u, B)(t)\|_{H^2}
:=\|\varrho(t)\|_{H^2}+\|u(t)\|_{H^2}+\| B(t)\|_{H^2} \le \delta,
\end{equation}
which, together with Sobolev inequality, yields directly
$$
\frac{1}{2}\le \varrho+1 \le \frac{3}{2}.
$$
Hence, we immediately have
\begin{equation}\label{2.6}
|f(\varrho)|, |h(\varrho)| \le C|\varrho| \quad
{\text{and}} \quad
|g^{(k-1)}(\varrho)|, |h^{(k)}(\varrho)|, |f^{(k)}(\varrho)| \le C
\quad {\text{for any}} ~ k \ge 1,
\end{equation}
which will be used frequently to derive a priori estimates.

We state the classical Sobolev interpolation of the Gagliardo-Nirenberg inequality, refer to \cite{Nirenberg}.
\begin{lemm}\label{lemma2.1}
Let $0 \le m, \alpha \le l$ and the function $f\in C_0^\infty(\mathbb{R}^3)$, then we have
\begin{equation}\label{GN}
\|\nabla^\alpha f\|_{L^p} \lesssim \|\nabla^m f \|_{L^2}^{1-\theta} \|\nabla^l f \|_{L^2}^\theta,
\end{equation}
where $0\le \theta \le 1$ and $\alpha$ satisfy
$$
\frac{1}{p}-\frac{\alpha}{3}=\left(\frac{1}{2}-\frac{m}{3}\right)(1-\theta)
+\left(\frac{1}{2}-\frac{l}{3}\right)\theta.
$$
\end{lemm}

First of all, we will derive following energy estimates.

\begin{lemm}\label{lemma2.2}
Under the condition \eqref{2.5}, then for $k=0, 1$, we have
\begin{equation}\label{221}
\frac{d}{dt} \|\nabla^k (\varrho, u, B)\|_{L^2}^2 +C\|\nabla^{k+1} (u, B) \|_{L^2}^2
\lesssim \delta \|\nabla^{k+1} \varrho\|_{L^2}^2.
\end{equation}
\end{lemm}
\begin{proof}
Taking $k$-th spatial derivatives to \eqref{eq1}$_1$ and \eqref{eq1}$_2$ respectively,
multiplying the resulting identities by $\nabla^k \varrho$ and $ \nabla^k u$
respectively and integrating over $\mathbb{R}^3$(by parts), it is easy to obtain
\begin{equation}\label{222}
\begin{aligned}
&\frac{1}{2}\frac{d}{dt}\int( |\nabla^k \varrho|^2+\!|\nabla^k u|^2) dx
+\int( \mu |\nabla^{k+1} u|^2+(\mu+\nu)|\nabla^k {\rm div} u|^2) dx\\
&=\int \nabla^k S_1 \cdot \nabla^k \varrho \ dx
+\int \nabla^k S_2 \cdot \nabla^k u \ dx.
\end{aligned}
\end{equation}
Taking $k$-th spatial derivatives to \eqref{eq1}$_3$, multiplying the resulting identity
by $\nabla^k B$ and integrating over $\mathbb{R}^3$(by parts), we have
\begin{equation}\label{223}
\frac{1}{2}\frac{d}{dt}\int |\nabla^{k} B|^2 \! dx
+\int |\nabla^{k+1} B|^2dx
=\int \nabla^{k}S_3 \cdot \nabla^{k}B ~ dx.
\end{equation}
Adding \eqref{222} to \eqref{223}, it follows immediately
\begin{equation}\label{224}
\begin{aligned}
&\!\frac{1}{2}\frac{d}{dt}\!\!\int \! (|\nabla^k \varrho|^2\!+\!|\nabla^k u|^2 \!+\!|\nabla^{k} B|^2)dx
+\!\!\int \! (\mu |\nabla^{k+1} u|^2\!+\!(\mu+\!\nu)|\nabla^k {\rm div} u|^2\!
+\!|\nabla^{k+1} B|^2)dx\\
&=\int \nabla^k S_1 \cdot \nabla^k \varrho \ dx
+\int \nabla^k S_2 \cdot \nabla^k u \ dx
+\int \nabla^{k}S_3 \cdot \nabla^{k}B~ dx.
\end{aligned}
\end{equation}
For the case $k=0$, then the differential identity \eqref{224} has the following form
\begin{equation}\label{225}
\begin{aligned}
&\frac{1}{2}\frac{d}{dt}\int ( |\varrho|^2+\!|u|^2+|B|^2 ) dx
 +\int ( \mu |\nabla u|^2+(\mu+\nu)|{\rm div} u|^2+|\nabla B|^2)dx\\
&=\int S_1 \cdot \varrho \ dx +\int S_2 \cdot u \ dx+\int S_3 \cdot B \ dx=I_1+I_2+I_3.\\
\end{aligned}
\end{equation}
Applying \eqref{2.5}, Holder, Sobolev and Young inequalities, it is easy to obtain
\begin{equation}\label{226}
\begin{aligned}
I_1
&\le \|\varrho \|_{L^3} \|{\rm div} u\|_{L^2} \|\varrho\|_{L^6}+\|\varrho\|_{L^3}\|\nabla \varrho\|_{L^2}\|u\|_{L^6}\\
&\lesssim \|\varrho\|_{H^1}\|\nabla u\|_{L^2}\|\nabla \varrho \|_{L^2}
          +\|\varrho\|_{H^1}\|\nabla \varrho\|_{L^2}\|\nabla u\|_{L^2}\\
&\lesssim \delta (\|\nabla \varrho\|_{L^2}^2+ \|\nabla u\|_{L^2}^2).
\end{aligned}
\end{equation}
Integrating by parts and applying \eqref{2.6}, Holder, Sobolev and Young inequalities,
it arrives at directly
\begin{equation}\label{227}
\begin{aligned}
&-\int h(\varrho)(\mu \Delta u +(\mu+\nu)\nabla {\rm div}u)udx\\
&\approx\int (h'(\varrho)\nabla \varrho \cdot u+h(\varrho)\nabla u)\nabla u dx\\
&\lesssim \|\nabla \varrho\|_{L^2}\|u\|_{L^6}\|\nabla u\|_{L^3}+\|\varrho\|_{L^\infty}\|\nabla u\|_{L^2}^2\\
&\lesssim (\|\varrho\|_{H^2}+\|\nabla u\|_{H^1})(\|\nabla \varrho\|_{L^2}^2+\|\nabla u\|_{L^2}^2)\\
&\lesssim \delta (\|\nabla \varrho\|_{L^2}^2+\|\nabla u\|_{L^2}^2).\\
\end{aligned}
\end{equation}
Hence, with the help of \eqref{2.6}, \eqref{227}, Holder, Sobolev and
Young inequalities, we deduce
\begin{equation}\label{228}
\begin{aligned}
I_2
&\lesssim (\|u\|_{L^3}\|\nabla u\|_{L^2}+\|\varrho \|_{L^3}\|\nabla \varrho\|_{L^2}
            +\| g(\varrho)\|_{L^\infty}\| B\|_{L^3}\|\nabla B\|_{L^2})\|u\|_{L^6}\\
&\quad \   +\delta (\|\nabla \varrho\|_{L^2}^2+\|\nabla u\|_{L^2}^2)\\
&\lesssim (\|u\|_{H^1}\|\nabla u\|_{L^2}+\|\varrho\|_{H^1}\|\nabla \varrho\|_{L^2}
            +\|B\|_{H^1}\|\nabla B\|_{L^2})\|\nabla u\|_{L^2}\\
&\quad \   +\delta (\|\nabla \varrho\|_{L^2}^2+\|\nabla u\|_{L^2}^2)\\
&\lesssim \delta(\|\nabla \varrho\|_{L^2}^2+\|\nabla u\|_{L^2}^2+\|\nabla B\|_{L^2}^2).
\end{aligned}
\end{equation}
Integrating by part and applying \eqref{2.6}, Holder and Sobolev inequalities,
it arrives at
\begin{equation}\label{229}
\begin{aligned}
&-\int {\rm curl}\left[g(\varrho)(B\cdot \nabla B)\right]Bdx\\
&=-\int g(\varrho)(B\cdot \nabla B){\rm curl}Bdx\\
&\lesssim \|g(\varrho)\|_{L^\infty}\|B\|_{L^\infty}\|\nabla B\|_{L^2}\|{\rm curl} B\|_{L^2}\\
&\lesssim \|B\|_{H^2}\|\nabla B\|_{L^2}^2.
\end{aligned}
\end{equation}
Hence, with the help of \eqref{229}, Holder, Sobolev and Young inequalities, we deduce
\begin{equation}\label{2210}
\begin{aligned}
I_3
&\lesssim (\|u\|_{L^3}\|\nabla B\|_{L^2}+\|B\|_{L^3}\|\nabla u\|_{L^2})\|B\|_{L^6}
           +\|B\|_{H^2}\|\nabla B\|_{L^2}^2\\
&\lesssim (\|u\|_{H^1}+\|B\|_{H^1})(\|\nabla u\|_{L^2}^2+\|\nabla B\|_{L^2}^2)
           +\delta\|\nabla B\|_{L^2}^2\\
&\lesssim \delta(\|\nabla u\|_{L^2}^2+\|\nabla B\|_{L^2}^2).
\end{aligned}
\end{equation}
Substituting \eqref{226}, \eqref{228} and \eqref{2210} into \eqref{225}
and applying the smallness of $\delta$, it arrives at directly
\begin{equation}\label{2211}
\frac{d}{dt}\int ( |\varrho|^2+\!|u|^2+|B|^2 ) dx+\int ( \mu |\nabla u|^2+|\nabla B|^2)dx
\lesssim \delta \|\nabla \varrho\|_{L^2}^2.
\end{equation}
For the case $k=1$, then the differential identity \eqref{224} has the following form
\begin{equation}\label{2212}
\begin{aligned}
&\frac{1}{2}\frac{d}{dt}\int ( |\nabla \varrho|^2+\!|\nabla u|^2+|\nabla B|^2 ) dx
 +\int ( \mu |\nabla^2 u|^2+(\mu+\nu)|\nabla{\rm div} u|^2+|\nabla^2 B|^2)dx\\
&=\int \nabla S_1 \cdot \nabla \varrho \ dx
+\int \nabla S_2 \cdot \nabla u \ dx
+\int \nabla S_3 \cdot \nabla B \ dx=I\!I_1+I\!I_2+I\!I_3.\\
\end{aligned}
\end{equation}
Applying Holder, Sobolev and Young inequalities, we obtain
\begin{equation}\label{2213}
\begin{aligned}
I\!I_1
&\le (\|\varrho\|_{L^3}\|{\rm div}u\|_{L^6}+\|u\|_{L^3}\|\nabla \varrho \|_{L^6})\|\nabla^2 \varrho\|_{L^2}\\
&\lesssim (\|\varrho\|_{H^1}+\|u\|_{H^1})(\|\nabla^2 \varrho\|_{L^2}^2+\|\nabla^2 u\|_{L^2}^2)\\
&\lesssim \delta (\|\nabla^2 \varrho\|_{L^2}^2+\|\nabla^2 u\|_{L^2}^2).
\end{aligned}
\end{equation}
Similarly, it is easy to deduce
\begin{equation}\label{2214}
\begin{aligned}
I\!I_2
&\le (\|u\|_{L^3}\|\nabla u\|_{L^6}+\|h(\varrho)\|_{L^\infty}\|\nabla^2 u \|_{L^2})\|\nabla^2 u\|_{L^2}\\
&\quad +(\|f(\varrho)\|_{L^3}\|\nabla \varrho\|_{L^6}
         +\|g(\varrho)\|_{L^\infty}\|B\|_{L^3}\|\nabla B \|_{L^6})\|\nabla^2 u\|_{L^2}\\
&\lesssim (\|\varrho\|_{H^2}+\|u\|_{H^1}+\|B\|_{H^1})
          (\|\nabla^2 \varrho\|_{L^2}^2+\|\nabla^2 u\|_{L^2}^2+\|\nabla^2 B\|_{L^2}^2)\\
&\lesssim \delta (\|\nabla^2 \varrho\|_{L^2}^2+\|\nabla^2 u\|_{L^2}^2+\|\nabla^2 B\|_{L^2}^2).
\end{aligned}
\end{equation}
Integrating by part and applying \eqref{2.6}, Holder and Sobolev inequalities, we obtain
\begin{equation}\label{2215}
\begin{aligned}
&-\int \nabla{\rm curl}[g(\varrho) (B\cdot \nabla B) ]\nabla B~dx\\
&=-\int \nabla[g(\varrho) (B\cdot \nabla B) ]\nabla {\rm curl} B~dx\\
&\lesssim (\|\nabla g(\varrho)\|_{L^6}\|B\|_{L^6}\|\nabla B\|_{L^6}
          +\|g(\varrho)\|_{L^\infty}\|\nabla B\|_{L^3}\|\nabla B\|_{L^6})
          \|\nabla {\rm curl} B\|_{L^2}\\
&\quad    +\|g(\varrho)\|_{L^\infty}\|B\|_{L^\infty}\|\nabla^2 B\|_{L^2}\|\nabla{\rm curl} B\|_{L^2}\\
&\lesssim (\|\nabla^2 \varrho\|_{L^2}\|\nabla B\|_{L^2}+\|\nabla B\|_{L^3}+\|B\|_{H^2})
           \|\nabla^2 B\|_{L^2}^2.
\end{aligned}
\end{equation}
Applying \eqref{2215}, Holder, Sobolev and Young inequalities, it arrives at directly
\begin{equation}\label{2216}
\begin{aligned}
I\!I_3
&\lesssim (\| u\|_{L^3}\|\nabla B\|_{L^6}+\|B\|_{L^3}\|\nabla u\|_{L^6})\|\nabla^2 B\|_{L^2}\\
&\quad  +(\|\nabla^2 \varrho\|_{L^2}\|\nabla B\|_{L^2}+\|\nabla B\|_{L^3}+\|B\|_{H^2})
           \|\nabla^2 B\|_{L^2}^2\\
&\lesssim \delta (\|\nabla^2 u\|_{L^2}^2+\|\nabla^2 B\|_{L^2}^2).
\end{aligned}
\end{equation}
Substituting \eqref{2213}, \eqref{2214} and \eqref{2216} into \eqref{2212}, then we obtain
$$
\frac{d}{dt}\int ( |\nabla \varrho|^2+\!|\nabla u|^2+|\nabla B|^2 ) dx
 +\int ( \mu |\nabla^2 u|^2+|\nabla^2 B|^2)dx
\lesssim \delta \|\nabla^2 \varrho\|_{L^2}^2,
$$
which, together with \eqref{2211}, completes the proof of the lemma.
\end{proof}
Next, we derive the second type of energy estimates involving the higher-order spatial derivatives of $\varrho$ and $u$.

\begin{lemm}\label{lemma2.3}
Under the condition \eqref{2.5}, then we have
\begin{equation}\label{231}
\frac{d}{dt}\|\nabla^2(\varrho, u, B)\|_{L^2}^2+C\|\nabla^3 (u, B)\|_{L^2}^2
\lesssim \delta \|\nabla^2 \varrho\|_{L^2}^2.
\end{equation}
\end{lemm}
\begin{proof}
Taking $k=2$ specially in \eqref{224}, we deduce immediately
\begin{equation}\label{232}
\begin{aligned}
&\!\frac{1}{2}\frac{d}{dt}\!\!\int \! (|\nabla^2 \varrho|^2\!+\!|\nabla^2 u|^2 \!+\!|\nabla^2 B|^2)dx
+\!\!\int \! (\mu |\nabla^3 u|^2\!+\!(\mu+\!\nu)|\nabla^2 {\rm div} u|^2\!
+\!|\nabla^3 B|^2)dx\\
&=\int \nabla^2 S_1 \cdot \nabla^2 \varrho ~ dx
 +\int \nabla^2 S_2 \cdot \nabla^2 u ~ dx
 +\int \nabla^2 S_3 \cdot \nabla^{2} B ~ dx.
\end{aligned}
\end{equation}
Applying Holder, Sobolev and Young inequalities, it is easy to obtain
\begin{equation}\label{233}
\begin{aligned}
&-\int \nabla^2 (\varrho {\rm div} u) \nabla^2 \varrho dx\\
&=-\int(\nabla^2 \varrho {\rm div}u+2\nabla \varrho \nabla {\rm div}u+\varrho \nabla^2{\rm div}u)\nabla^2 \varrho dx\\
&\lesssim (\|\nabla u\|_{L^\infty}\|\nabla^2 \varrho\|_{L^2}+\|\nabla \varrho\|_{L^3}\|\nabla^2 u\|_{L^6}
      +\|\varrho\|_{L^\infty}\|\nabla^3 u\|_{L^2})\|\nabla^2 \varrho\|_{L^2}\\
&\lesssim (\|\nabla^2 u\|_{L^2}^{\frac{1}{2}}\|\nabla^3 u\|_{L^2}^{\frac{1}{2}}\|\nabla^2 \varrho\|_{L^2}
           +\|\nabla \varrho\|_{H^1}\|\nabla^3 u\|_{L^2}
           +\|\varrho\|_{H^2}\|\nabla^3 u\|_{L^2})\|\nabla^2 \varrho\|_{L^2}\\
&\lesssim (\|\nabla^2 u\|_{L^2}^{\frac{1}{2}}\|\nabla^2 \varrho\|_{L^2}^{\frac{1}{2}}
           +\|\nabla \varrho\|_{H^1}+\|\varrho\|_{H^2})(\|\nabla^2 \varrho\|_{L^2}^2+\|\nabla^3 u\|_{L^2}^2)\\
&\lesssim \delta(\|\nabla^2 \varrho\|_{L^2}^2+\|\nabla^3 u\|_{L^2}^2).
\end{aligned}
\end{equation}
Integrating by part and applying Holder, Sobolev and Young inequalities, it arrives at
\begin{equation}\label{234}
\begin{aligned}
&-\int\! \nabla^2 (u \cdot \nabla \varrho)\! \nabla^2 \varrho \ dx\\
&=\int \left[-(\nabla^2u \nabla \varrho+2\nabla u \nabla^2 \varrho)\nabla^2 \varrho
+\frac{1}{2}|\nabla^2 \varrho|^2{\rm div}u\right] dx\\
&\lesssim (\|\nabla^2 u\|_{L^6}\|\nabla \varrho\|_{L^3}+\|\nabla u\|_{L^\infty}\|\nabla^2 \varrho\|_{L^2})
           \|\nabla^2 \varrho\|_{L^2}\\
&\lesssim \|\nabla \varrho\|_{H^1}\|\nabla^2 \varrho\|_{L^2}\|\nabla^3 u\|_{L^2}
          +\|\nabla^2 u\|_{L^2}^{\frac{1}{2}}\|\nabla^3 u\|_{L^2}^{\frac{1}{2}}\|\nabla^2 \varrho\|_{L^2}^2\\
&\lesssim (\|\nabla^2 u\|_{L^2}^{\frac{1}{2}}\|\nabla^2 \varrho\|_{L^2}^{\frac{1}{2}}
           +\|\nabla \varrho\|_{H^1})(\|\nabla^2 \varrho\|_{L^2}^2+\|\nabla^3 u\|_{L^2}^2)\\
&\lesssim \delta(\|\nabla^2 \varrho\|_{L^2}^2+\|\nabla^3 u\|_{L^2}^2).
\end{aligned}
\end{equation}
The combination of \eqref{233} and \eqref{234} gives rise to
\begin{equation}\label{235}
\int \nabla^2 S_1 \cdot \nabla^2 \varrho \ dx \lesssim \delta(\|\nabla^2 \varrho\|_{L^2}^2+\|\nabla^3 u\|_{L^2}^2).
\end{equation}
Now, we give the estimate for the second term on the right hand side of \eqref{232}.
By virtue of Holder and Sobolev inequalities, we have
\begin{equation}\label{236}
\begin{aligned}
&\int \nabla (u\cdot \nabla u)\nabla \Delta udx\\
&=\int (\nabla u \nabla u+u\nabla^2 u)\nabla \Delta udx\\
&\le \|\nabla u\|_{L^3}\|\nabla u\|_{L^6}\|\nabla^3 u\|_{L^2}
+\|u\|_{L^3}\|\nabla^2 u\|_{L^6}\|\nabla^3 u\|_{L^2}\\
&\lesssim \|u\|_{L^2}^{\frac{1}{2}}\|\nabla^3 u\|_{L^2}^{\frac{1}{2}}
          \|\nabla u\|_{L^2}^{\frac{1}{2}}\|\nabla^3 u\|_{L^2}^{\frac{1}{2}}\|\nabla^3 u\|_{L^2}
          +\|u\|_{H^1}\|\nabla^3 u\|_{L^2}^2\\
&\lesssim \delta \|\nabla^3 u\|_{L^2}^2.\\
\end{aligned}
\end{equation}
In view of \eqref{2.6}, Holder and Sobolev inequalities, we have
\begin{equation}\label{237}
\begin{aligned}
&\int \nabla(h(\varrho)(\mu \Delta u+(\mu+\nu)\nabla {\rm div}u))\nabla \Delta u dx\\
&\lesssim (\|\nabla h(\varrho)\|_{L^3}\|\nabla^2 u\|_{L^6}+\|h(\varrho)\|_{L^\infty}\|\nabla^3 u\|_{L^2})\|\nabla^3 u\|_{L^2}\\
&\lesssim (\|\nabla \varrho\|_{H^1}\|\nabla^3 u\|_{L^2}+\|\varrho\|_{H^2}\|\nabla^3 u\|_{L^2})\|\nabla^3 u\|_{L^2}\\
&\lesssim \delta \|\nabla^3 u\|_{L^2}^2,
\end{aligned}
\end{equation}
and
\begin{equation}\label{238}
\begin{aligned}
&\int \nabla(f(\varrho)\nabla \varrho)\nabla \Delta u dx\\
&\lesssim (\|\nabla \varrho\|_{L^4}^2
+\|f(\varrho)\|_{L^\infty}\|\nabla^2 \varrho\|_{L^2})\|\nabla^3 u\|_{L^2}\\
&\lesssim (\|\nabla^{\frac{3}{2}} \varrho\|_{L^2}\|\nabla^2 \varrho\|_{L^2}
           +\|\varrho\|_{H^2}\|\nabla^2 \varrho\|_{L^2})\|\nabla^3 u\|_{L^2}\\
&\lesssim (\|\nabla^{\frac{3}{2}} \varrho\|_{L^2}+\|\varrho\|_{H^2})
            \|\nabla^2\varrho\|_{L^2}\|\nabla^3 u\|_{L^2}\\
&\lesssim \delta (\|\nabla^2 \varrho\|_{L^2}^2+\|\nabla^3 u\|_{L^2}^2).
\end{aligned}
\end{equation}
Similarly, it is easy to deduce
\begin{equation}\label{239}
\begin{aligned}
&\int \nabla\left[g(\varrho)(B \cdot \nabla B-\nabla(\frac{1}{2}|B|^2))\right]\nabla \Delta u dx\\
&\lesssim (\|\nabla g(\varrho)\|_{L^6}\|B\|_{L^6}\|\nabla B\|_{L^6}
           +\|g(\varrho)\|_{L^\infty}\|\nabla B\|_{L^3}\|\nabla B\|_{L^6})\|\nabla^3 u\|_{L^2}\\
&\quad  +\|g(\varrho)\|_{L^\infty}\|B\|_{L^3}\|\nabla^2 B\|_{L^6}\|\nabla^3 u\|_{L^2}\\
&\lesssim (\|\nabla \varrho\|_{L^6}\|\nabla B\|_{L^2}\|\nabla^2 B\|_{L^2}
          +\|B\|_{L^2}^{\frac{1}{2}}\|\nabla^3 B\|_{L^2}^{\frac{1}{2}}
           \|\nabla B\|_{L^2}^{\frac{1}{2}}\|\nabla^3 B\|_{L^2}^{\frac{1}{2}})\|\nabla^3 u\|_{L^2}\\
&\quad     +\|B\|_{H^1}\|\nabla^3 B\|_{L^2}\|\nabla^3 u\|_{L^2}\\
&\lesssim \delta (\|\nabla^2 \varrho\|_{L^2}^2+\|\nabla^3 u\|_{L^2}^2+\|\nabla^3 B\|_{L^2}^2).
\end{aligned}
\end{equation}
By virtue of the estimates \eqref{236}-\eqref{239}, we obtain immediately
\begin{equation}\label{2310}
\int \nabla^2 S_2 \cdot \nabla^2 u \ dx
\lesssim \delta (\|\nabla^2 \varrho\|_{L^2}^2+\|\nabla^3 u\|_{L^2}^2+\|\nabla^3 B\|_{L^2}^2).
\end{equation}
Integrating by part and applying Holder and Sobolev inequalities, it arrives at
\begin{equation}\label{2311}
\begin{aligned}
&\int \nabla^2(u\cdot \nabla B)\nabla^2 B~dx\\
&\lesssim (\|\nabla u\|_{L^3}\|\nabla B\|_{L^6}
           +\|u\|_{L^3}\|\nabla^2 B\|_{L^6})\|\nabla^3 B\|_{L^2}\\
&\lesssim (\|u\|_{L^2}^{\frac{1}{2}}\|\nabla^3 u\|_{L^2}^{\frac{1}{2}}
           \|\nabla B\|_{L^2}^{\frac{1}{2}}\|\nabla^3 B\|_{L^2}^{\frac{1}{2}}
           +\|u\|_{H^1}\|\nabla^3 B\|_{L^2})\|\nabla^3 B\|_{L^2}\\
&\lesssim \delta(\|\nabla^3 u\|_{L^2}^2+\|\nabla^3 B\|_{L^2}^2).
\end{aligned}
\end{equation}
Similarly, it is easy to obtain
\begin{equation}\label{2312}
\begin{aligned}
&\int \nabla^2(B\cdot \nabla u+B{\rm div}u)\nabla^2 B~dx\\
&\lesssim (\|\nabla u\|_{L^3}\|\nabla B\|_{L^6}
           +\|B\|_{L^3}\|\nabla^2 u\|_{L^6})\|\nabla^3 B\|_{L^2}\\
&\lesssim (\|u\|_{L^2}^{\frac{1}{2}}\|\nabla^3 u\|_{L^2}^{\frac{1}{2}}
           \|\nabla B\|_{L^2}^{\frac{1}{2}}\|\nabla^3 B\|_{L^2}^{\frac{1}{2}}
           +\|B\|_{H^1}\|\nabla^3 u\|_{L^2})\|\nabla^3 B\|_{L^2}\\
&\lesssim \delta(\|\nabla^3 u\|_{L^2}^2+\|\nabla^3 B\|_{L^2}^2).
\end{aligned}
\end{equation}
Integrating by part and applying \eqref{2.6}, Holder and Sobolev inequalities, we obtain
\begin{equation}\label{2313}
\begin{aligned}
&\quad -\int \nabla^2{\rm curl}\left[g(\varrho)\left(B\cdot \nabla B
                 -\nabla(\frac{1}{2}|B|^2)\right)\right]\nabla^2 B~dx\\
&=-\int\nabla^2\left[g(\varrho)\left(B\cdot \nabla B-\nabla(\frac{1}{2}|B|^2)\right)\right]
                 \nabla^2{\rm curl}B~dx\\
&\lesssim (\|\nabla^2 g(\varrho)\|_{L^2}\|B\|_{L^\infty}\|\nabla B\|_{L^\infty}
          +\|\nabla g(\varrho)\|_{L^6}\|\nabla B\|_{L^6}\|\nabla B\|_{L^6})
          \|\nabla^2 {\rm curl} B\|_{L^2}\\
&\quad   +(\|\nabla g(\varrho)\|_{L^6}\|B\|_{L^6}\|\nabla^2 B\|_{L^6}
          +\|g(\varrho)\|_{L^\infty}\|\nabla B\|_{L^3}\|\nabla^2 B\|_{L^6})
          \|\nabla^2 {\rm curl} B\|_{L^2}\\
&\quad   +\|g(\varrho)\|_{L^\infty}\|B\|_{L^\infty}\|\nabla^3 B\|_{L^2}\|\nabla^2 {\rm curl} B\|_{L^2}\\
&\lesssim (\||\nabla \varrho|^2+|\nabla^2 \varrho|\|_{L^2}\|B\|_{L^\infty}\|\nabla B\|_{L^\infty}
          +\|\nabla \varrho\|_{L^6}\|\nabla B\|_{L^6}\|\nabla B\|_{L^6})
          \|\nabla^3 B\|_{L^2}\\
&\quad   +(\|\nabla \varrho\|_{L^6}\|B\|_{L^6}\|\nabla^2 B\|_{L^6}
          +\|\nabla B\|_{L^3}\|\nabla^2 B\|_{L^6}+\|B\|_{L^\infty}\|\nabla^3 B\|_{L^2})
          \|\nabla^3 B\|_{L^2}\\
&\lesssim (\|\nabla \varrho\|_{L^2}^\frac{1}{2}\|\nabla^2 \varrho\|_{L^2}^\frac{3}{2}
          +\|\nabla^2 \varrho\|_{L^2})
          \|B\|_{H^2}\|\nabla^2 B\|_{L^2}^\frac{1}{2}\|\nabla^3 B\|_{L^2}^\frac{3}{2}\\
&\quad   +\|\nabla B\|_{H^1}^2 \|\nabla^2 \varrho\|_{L^2} \|\nabla^3 B\|_{L^2}
         +\|\nabla B\|_{H^1}\|\nabla^3 B\|_{L^2}^2\\
&\lesssim \delta(\|\nabla^2 \varrho\|_{L^2}^2+\|\nabla^3 B\|_{L^2}^2).
\end{aligned}
\end{equation}
In view of the estimates \eqref{2311}-\eqref{2313}, we obtain directly
\begin{equation}\label{2314}
\int \nabla^2 S_3 \cdot \nabla^2 B ~ dx
\lesssim \delta(\|\nabla^2 \varrho\|_{L^2}^2+\|\nabla^3 u\|_{L^2}^2+\|\nabla^3 B\|_{L^2}^2).
\end{equation}
Substituting \eqref{235}, \eqref{2310} and \eqref{2314} into \eqref{232}, then we have
$$
\frac{d}{dt}\int ( |\nabla^2 \varrho|^2+\!|\nabla^2 u|^2+|\nabla^2 B|^2 ) dx
 +\int ( \mu |\nabla^3 u|^2+|\nabla^3 B|^2)dx
\lesssim \delta \|\nabla^2 \varrho\|_{L^2}^2,
$$
which completes the proof of the lemma.
\end{proof}

Finally, we will use the equations \eqref{eq1} to recover the dissipation estimate for $\varrho$.

\begin{lemm}\label{lemma2.4}
Under the condition \eqref{2.5}, then for $k=0, 1$, we have
\begin{equation}\label{241}
\frac{d}{dt}\int \nabla^k u \cdot \nabla^{k+1} \varrho dx+C\|\nabla^{k+1} \varrho\|_{L^2}^2
\lesssim  \|\nabla^{k+1} u\|_{L^2}^2+\|\nabla^{k+2}u\|_{L^2}^2+\|\nabla^{k+2} B\|_{L^2}^2.
\end{equation}
\end{lemm}
\begin{proof}
Taking $k$-th spatial derivatives to the second equation of \eqref{eq1}, multiplying by $\nabla^{k+1}\varrho$ and integrating over $\mathbb{R}^3$, then we obtain
\begin{equation}\label{242}
\begin{aligned}
&\int \nabla^k u_t \cdot \nabla^{k+1} \varrho dx+\int |\nabla^{k+1} \varrho|^2 dx\\
&=\int \nabla^k [\mu \Delta u+(\mu+\nu)\nabla {\rm div}u]\nabla^{k+1} \varrho dx
+ \int \nabla^k S_2 \cdot \nabla^{k+1} \varrho dx.
\end{aligned}
\end{equation}
In order to deal with $\int \nabla^k u_t \cdot \nabla^{k+1} \varrho dx$,
we turn the time derivatives of velocity to the density. Then, applying the mass
equation \eqref{eq1}$_1$, we can transform time derivatives to the spatial derivatives, i.e.
\begin{equation}\label{243}
\begin{aligned}
&\int \nabla^k u_t \cdot \nabla^{k+1} \varrho dx\\
&=\frac{d}{dt} \int \nabla^k u \cdot \nabla^{k+1} \varrho dx-\int \nabla^k u \cdot \nabla^{k+1} \varrho_t dx\\
&=\frac{d}{dt} \int \nabla^k u \cdot \nabla^{k+1} \varrho dx+\int \nabla^k u\cdot
     \nabla^{k+1} ({\rm div}u+{\rm div}(\varrho u))dx\\
&=\frac{d}{dt} \int \nabla^k u \cdot \nabla^{k+1} \varrho dx-\int \nabla^k {\rm div}u\cdot
     \nabla^{k} ({\rm div}u+{\rm div}(\varrho u))dx\\
&=\frac{d}{dt} \int \nabla^k u \cdot \nabla^{k+1} \varrho dx-\int |\nabla^k {\rm div}u|^2 dx
   -\int \nabla^k {\rm div}u\cdot\nabla^{k} {\rm div}(\varrho u)dx.
\end{aligned}
\end{equation}
Substituting \eqref{243} into \eqref{242}, it is easy to deduce
\begin{equation}\label{244}
\begin{aligned}
&\frac{d}{dt} \int \nabla^k u \cdot \nabla^{k+1} \varrho dx+\int |\nabla^{k+1} \varrho|^2 dx\\
&=\int\! |\nabla^k {\rm div}u|^2 dx
   +\!\int \!\nabla^k {\rm div}u\cdot\nabla^{k} {\rm div}(\varrho u)dx
   +\int \nabla^k S_2 \cdot \nabla^{k+1} \varrho dx\\
& \quad +\!\int \!\nabla^k [\mu \Delta u+(\mu+\nu)\nabla {\rm div}u]\nabla^{k+1} \varrho dx.
\end{aligned}
\end{equation}
For the case $k=0,$ then applying Holder, Sobolev and Young inequalities, we obtain
\begin{equation}\label{245}
\begin{aligned}
\int \! {\rm div}u\cdot {\rm div}(\varrho u)dx
&\lesssim \|\varrho\|_{L^\infty} \|\nabla u\|_{L^2}^2
          +\|u\|_{L^3}\|{\rm div}u\|_{L^6}\|\nabla \varrho\|_{L^2}\\
&\lesssim (\|\varrho\|_{H^2}+\|u\|_{H^1})(\|\nabla \varrho\|_{L^2}^2+\|\nabla^2 u\|_{L^2}^2)\\
&\lesssim \delta (\|\nabla \varrho\|_{L^2}^2+\|\nabla^2 u\|_{L^2}^2).
\end{aligned}
\end{equation}
By virtue of \eqref{2.6} and Holder inequality, it is easy to deduce
\begin{equation}\label{246}
\begin{aligned}
\int  S_2 \cdot \nabla \varrho dx
&\lesssim (\|u\|_{L^3}\|\nabla u\|_{L^6}+\|\varrho\|_{L^\infty}\|\nabla^2 u \|_{L^2})\|\nabla \varrho\|_{L^2}\\
& \quad    +(\|\varrho\|_{L^\infty}\|\nabla \varrho\|_{L^2}
           +\|g(\varrho)\|_{L^\infty}\|B\|_{L^3}\|\nabla B \|_{L^6})\|\nabla \varrho\|_{L^2}\\
&\lesssim \delta (\|\nabla \varrho\|_{L^2}^2+\|\nabla^2 u\|_{L^2}^2+\|\nabla^2 B\|_{L^2}^2),
\end{aligned}
\end{equation}
and
\begin{equation}\label{247}
\begin{aligned}
\int  [\mu \Delta u+(\mu+\nu)\nabla {\rm div}u]\nabla \varrho dx
\lesssim \|\nabla^2 u\|_{L^2}^2  + \varepsilon\|\nabla \varrho\|_{L^2}^2.
\end{aligned}
\end{equation}
The combination of \eqref{245}, \eqref{246} and \eqref{247} helps us complete the proof to \eqref{241} for the case of $k=0$. As for the case $k=1,$ applying Holder, Sobolev and Young inequalities, we deduce
\begin{equation}\label{248}
\begin{aligned}
&\int \nabla {\rm div}u\cdot\nabla {\rm div}(\varrho u)dx\\
&\lesssim (\|\nabla \varrho\|_{L^3}\|{\rm div} u\|_{L^6}
           +\|\varrho\|_{L^\infty}\|\nabla {\rm div}u\|_{L^2})\|\nabla^2 u\|_{L^2}\\
&\qquad  +(\|\nabla \varrho\|_{L^3}\|\nabla u\|_{L^6}+\|u\|_{L^\infty}\|\nabla^2 \varrho\|_{L^2})\|\nabla^2 u\|_{L^2}\\
&\lesssim \delta (\|\nabla^2 \varrho\|_{L^2}^2+\|\nabla^2 u\|_{L^2}^2).
\end{aligned}
\end{equation}
With the help of Holder inequality and Lemma \ref{lemma2.3}, it arrives at
\begin{equation}\label{249}
\int \nabla S_2 \cdot \nabla^2 \varrho dx
\lesssim \delta (\|\nabla^2 \varrho\|_{L^2}^2+\|\nabla^3 u\|_{L^2}^2+\|\nabla^3 B\|_{L^2}^2),
\end{equation}
and
\begin{equation}\label{2410}
\begin{aligned}
\int  \nabla[\mu \Delta u+(\mu+\nu)\nabla {\rm div}u]\nabla^2 \varrho dx
\lesssim \|\nabla^3 u\|_{L^2}^2+\varepsilon\|\nabla^2 \varrho\|_{L^2}^2.
\end{aligned}
\end{equation}
The combination of \eqref{248}, \eqref{249} and \eqref{2410} gives rise to
the proof of \eqref{241} for the case of $k=1$.
\end{proof}

\subsection{Global existence of strong solutions}
\quad In this subsection, we shall combine the energy estimates that we have derived
in the previous section to prove the global existence of strong solutions.
Summing up \eqref{221} from $k=l \ (l=0,1)$ to $k=1$, then we obtain
$$
\frac{d}{dt}\|\nabla^l(\varrho, u, B)\|_{H^{1-l}}^2
+C\|\nabla^l(\nabla u, \nabla B)\|_{H^{1-l}}^2
\lesssim \delta \|\nabla^{l+1} \varrho\|_{H^{1-l}}^2,
$$
which, together with \eqref{231}, then it arrives at
\begin{equation}\label{251}
\frac{d}{dt}\|\nabla^l(\varrho, u, B)\|_{H^{2-l}}^2+C\|\nabla^{l+1}(u, B)\|_{H^{2-l}}^2
\le \delta C_1\|\nabla^{l+1} \varrho\|_{H^{1-l}}^2.
\end{equation}
On the other hand, summing \eqref{241} from $k=l \ (l=0,1)$ to $k=1$, we obtain immediately
\begin{equation}\label{252}
\frac{d}{dt}\sum_{l\le k \le 1}\int \nabla^k u \cdot \nabla^{k+1} \varrho dx
+C_2\|\nabla^{l+1} \varrho\|_{H^{1-l}}^2
\le C\left(\|\nabla^{l+1} u\|_{H^{2-l}}^2+\|\nabla^{l+2} B\|_{H^{1-l}}^2\right).
\end{equation}
Multiplying \eqref{252} by $2\delta C_1/C_2$ and adding the resulting inequality
to \eqref{251}, then it arrives at
\begin{equation}\label{253}
\frac{d}{dt}\mathcal{E}_l^2(t)+
C_3 \left(\|\nabla^{l+1} \varrho\|_{H^{1-l}}^2+\|\nabla^{l+1}(u, B)\|_{H^{2-l}}^2\right)\le 0.
\end{equation}
where  $\mathcal{E}_l^2(t)$ is defined as
$$
\mathcal{E}_l^2(t)=
\|\nabla^l(\varrho, u, B)\|_{H^{2-l}}^2
+\frac{2\delta C_1}{C_2}\sum_{l\le k \le 1}\int \nabla^k u \cdot \nabla^{k+1} \varrho dx.
$$
By virtue of the smallness of $\delta$, it is easy to obtain
\begin{equation}\label{254}
C_4^{-1}\|\nabla^l(\varrho, u, B)\|_{H^{2-l}}^2
\le
\mathcal{E}_l^2(t)
\le C_4\|\nabla^l(\varrho, u, B)\|_{H^{2-l}}^2.
\end{equation}
Choosing $l=0$ in \eqref{253}, integrating over $[0, t]$ and applying the
equivalent relation \eqref{254}, we obtain
$$
\|(\varrho, u, B)(t)\|_{H^2}^2 \le C \|(\varrho_0, u_0, B_0)\|_{H^2}^2.
$$
Then, by the standard continuity argument (see Theorem 7.1 on page 100 in
\cite{Matsumura-Nishida}), we close the estimate \eqref{2.5}.
Thus, we extend the local strong solutions to be global one
and the uniqueness of global strong solutions is guaranteed by the uniqueness
of local solutions that has been prove by Fan et al. \cite{Fan-Zhou}.
Therefore,
choosing $l=0$ in \eqref{253}, integrating over $[0, t]$ and applying the equivalent
relation \eqref{254}, we obtain
\begin{equation}\label{255}
\|(\varrho, u, B)(t)\|_{H^2}^2+
 \int_0^t (\|\nabla \varrho (\tau)\|_{H^{1}}^2+\|(\nabla u, \nabla B)(\tau)\|_{H^2}^2) d\tau
\le C \|(\varrho_0, u_0, B_0)\|_{H^2}^2,
\end{equation}
which completes the proof of  the global existence of strong solutions.

\subsection{Decay rates of strong solution}

\quad In this section, we will establish optimal decay rates for the compressible
Hall-MHD equations \eqref{1.1}-\eqref{1.3}.
If the initial perturbation belongs to $L^1$ additionally, we apply the method
of Green function to establish optimal time decay rates for the global strong solutions.
Furthermore, the application of Fourier splitting method by Schonbek \cite{Schonbek}
helps us to build optimal time decay rates for the second order spatial derivatives
of magnetic field.

First of all, let us to consider the following linearized systems
\begin{equation}\label{leq1}
\left\{
\begin{aligned}
& \varrho_t+{\rm div}u=0,\\
& u_t-\mu \Delta u-(\mu+\nu)\nabla {\rm div}u+\nabla \varrho=0,\\
& B_t-\Delta B=0,
\end{aligned}
\right.
\end{equation}
with the initial data
\begin{equation}\label{leq2}
\left.(\varrho, u, B)(x,t)\right|_{t=0}=(\varrho_0, u_0, B_0)(x)
\rightarrow(0,0,0) \quad {\text {as} } \quad |x|\rightarrow \infty.
\end{equation}
Obviously, the solution $(\varrho, u, B)$ of the linearized problem \eqref{leq1}-\eqref{leq2}
can be expressed as
\begin{equation}\label{leqs}
(\varrho, u, B)^{tr}=G(t)*(\varrho_0, u_0, B_0)^{tr}, t \ge 0.
\end{equation}
Here $G(t):=G(x,t)$ is the Green matrix for the systems \eqref{leq1}
and the exact expression of the Fourier transform $\hat{G}(\xi, t)$ of Green function $G(x,t)$ as
$$
\hat{G}(\xi, t)=
\left[
\begin{array}{ccc}
\frac{\lambda_+ e^{\lambda_- t}-\lambda_- e^{\lambda_+ t}}{\lambda_+ -\lambda_-}
& \frac{-i \xi^t (e^{\lambda_+ t}-e^{\lambda_- t})}{\lambda_+ -\lambda_-}
& 0\\
 \frac{-i \xi (e^{\lambda_+ t}-e^{\lambda_- t})}{\lambda_+ -\lambda_-}
& \frac{\lambda_+ e^{\lambda_+ t}-\lambda_- e^{\lambda_- t}}{\lambda_+ -\lambda_-} \frac{\xi \xi^t}{|\xi|^2}
  +e^{\lambda_0 t}\left(I_{3\times 3}- \frac{\xi \xi^t}{|\xi|^2}\right)
&0 \\
0
&0
&e^{\lambda_1 t}I_{3\times 3}\\
\end{array}
\right]
$$
where
$$
\begin{aligned}
&\lambda_0 =-\mu |\xi|^2,\quad \lambda_1 =-|\xi|^2,\\
&\lambda_+ =-\left(\mu+\frac{1}{2}\nu\right)|\xi|^2+ i \sqrt{|\xi|^2-\left(\mu+\frac{1}{2}\nu\right)^2|\xi|^4},\\
&\lambda_- =-\left(\mu+\frac{1}{2}\nu\right)|\xi|^2- i \sqrt{|\xi|^2-\left(\mu+\frac{1}{2}\nu\right)^2|\xi|^4}.\\
\end{aligned}
$$
Since the systems \eqref{261} is an independent coupling of the classical linearized Navier-Stokes equations and heat equation, the representation of Green function $\hat{G}(\xi, t)$ is easy to verify. Furthermore, we have the following decay rates for the systems \eqref{leq1}-\eqref{leq2},
refer to \cite{Duan-Yang}.
\begin{prop}\label{linearized-Decay}
Assume that $(\varrho, u, B)$ is the solution of the linearized systems
\eqref{leq1}-\eqref{leq2} with the initial data
$(\varrho_0, u_0, B_0)\in L^1 \cap H^2 $, then
$$
\begin{aligned}
&\|\nabla^k \varrho\|_{L^2}^2\le C\left(\|(\varrho_0, u_0)\|_{L^1}^2+\|\nabla^k(\varrho_0, u_0)\|_{L^2}^2\right)
 (1+t)^{-\frac{3}{2}-k},\\
&\|\nabla^k u\|_{L^2}^2\le C\left(\|(\varrho_0, u_0)\|_{L^1}^2+\|\nabla^k(\varrho_0, u_0)\|_{L^2}^2\right)
 (1+t)^{-\frac{3}{2}-k},\\
&\|\nabla^k B\|_{L^2}^2\le C\left(\|B_0\|_{L^1}^2+\|\nabla^k B_0\|_{L^2}^2\right)
 (1+t)^{-\frac{3}{2}-k}
\end{aligned}
$$
for $0\le k \le 2$.
\end{prop}

In the sequel, we want to verify some estimates that play an important role
for us to derive decay rates for the compressible Hall-MHD equations
\eqref{eq1}-\eqref{eq4}.
\begin{equation}\label{non-estimates}
\begin{aligned}
\|(S_1, S_2, S_3)\|_{L^1}
&\lesssim \delta (\|\nabla \varrho\|_{L^2}+\|\nabla u\|_{H^1}+\|\nabla B\|_{H^1}),\\
\|(S_1, S_2, S_3)\|_{L^2}
&\lesssim \delta (\|\nabla \varrho\|_{L^2}+\|\nabla u\|_{H^1}+\|\nabla B\|_{H^1}),\\
\|\nabla (S_1, S_2, S_3)\|_{L^2}
&\lesssim \delta (\|\nabla^2 \varrho\|_{L^2}+\|\nabla^2 u\|_{L^2}+\|\nabla^2 B\|_{L^2})
+\|\nabla (\varrho, B)\|_{H^1}\|\nabla^2 (u, B)\|_{H^1}.\\
\end{aligned}
\end{equation}

Now, we establish the decay rates for the compressible Hall-MHD equations \eqref{eq1}-\eqref{eq4}.

\begin{lemm}\label{lemma2.6}
Under the assumptions of Theorem \ref{THM1}, the global strong solution $(\varrho, u, B)$
of problem \eqref{eq1}-\eqref{eq4} has the time decay rates
\begin{equation}\label{261}
\|\nabla^k \varrho (t)\|_{H^{2-k}}^2+\|\nabla^k u (t)\|_{H^{2-k}}^2
+\|\nabla^k B (t)\|_{H^{2-k}}^2
\le C(1+t)^{-\frac{3}{2}-k}
\end{equation}
for $k=0,1$.
\end{lemm}
\begin{proof}
Taking $l=1$ specially in \eqref{253}, it arrives at directly
\begin{equation}\label{262}
\frac{d}{dt}\mathcal{E}_1^2(t)+
C_3 \left(\|\nabla^{2} \varrho\|_{L^2}^2+\|\nabla^{2}u\|_{H^{1}}^2+\|\nabla^{2}B\|_{H^{1}}^2\right)\le 0,
\end{equation}
where $\mathcal{E}_1^2(t)$ is defined as
$$
\begin{aligned}
\mathcal{E}_1^2(t)=\|\nabla \varrho\|_{H^1}^2
+\|\nabla u\|_{H^1}^2+\|\nabla B\|_{H^1}^2
+\frac{2C_1\delta}{C_2}\int \nabla  u \cdot \nabla^2 \varrho dx.
\end{aligned}
$$
With the help of Young inequality and smallness of $\delta$, it is easy to deduce
\begin{equation}\label{263}
C_4^{-1} \|\nabla (\varrho, u, B)\|_{H^{1}}^2
\le
\mathcal{E}_1^2(t)
\le C_4 \|\nabla (\varrho, u, B)\|_{H^{1}}^2.
\end{equation}
Adding both hand sides of \eqref{262} by $\|\nabla (\varrho, u, B)\|_{L^2}^2$
and applying the equivalent relation \eqref{263}, then we have
$$
\frac{d}{dt}\mathcal{E}_1^2(t)+C\mathcal{E}^2_1(t)\le\|\nabla (\varrho, u, B)\|_{L^2}^2.
$$
In view of the Gronwall inequality, it follows immediately
\begin{equation}\label{264}
\begin{aligned}
\mathcal{E}_1^2(t)
&\le \mathcal{E}_1^2(0)e^{-Ct}
     +\int_0^t e^{-C(t-\tau)}\|\nabla (\varrho, u, B)(\tau)\|_{L^2}^2d\tau.
\end{aligned}
\end{equation}
In order to derive the time decay rate for $\mathcal{E}_1^2(t)$,  we need to control the term
$\|\nabla (\varrho, u, B)\|_{L^2}^2$. In fact, by Duhamel principle, we can represent the solutions
for the problem \eqref{eq1}-\eqref{eq4} as
\begin{equation}\label{265}
(\varrho, u, B)^{tr}(t)=G(t)*(\varrho_0, u_0, B_0)^{tr}+\int_0^t G(t-s)*(S_1,S_2, S_3)^{tr}(s) ds.
\end{equation}
Denoting
$$E(t)=\underset{0 \le \tau \le t}{\sup}
(1+\tau)^{\frac{5}{2}}(\|\nabla \varrho(\tau)\|_{H^{1}}^2+\|\nabla u(\tau)\|_{H^{1}}^2
+\|\nabla B(\tau)\|_{H^{1}}^2),$$
which, together with \eqref{265}, \eqref{non-estimates} and Proposition \ref{linearized-Decay},
gives directly
$$
\begin{aligned}
\|\nabla(\varrho, u, B)(t)\|_{L^2}
&\le C(1+t)^{-\frac{5}{4}}
  +C\int_0^t \left(\|(S_1, S_2, S_3)(\tau)\|_{L^1}
  +\|\nabla(S_1, S_2, S_3)(\tau)\|_{L^2}\right)(1+t-\tau)^{-\frac{5}{4}}d\tau\\
&\le C(1+t)^{-\frac{5}{4}}
  +C\int_0^t \delta\left(\|\nabla \varrho (\tau)\|_{H^1}+\|\nabla u (\tau)\|_{H^1}
  +\|\nabla B (\tau)\|_{H^1} \right)(1+t-\tau)^{-\frac{5}{4}}d\tau\\
&\quad \ +C\int_0^t
         \|\nabla (\varrho, B)(\tau)\|_{H^1}\|\nabla^2 (u, B)(\tau)\|_{H^1}
         (1+t-\tau)^{-\frac{5}{4}}d\tau\\
&\le C(1+t)^{-\frac{5}{4}}
  +C\delta \sqrt{E(t)}\int_0^t (1+t-\tau)^{-\frac{5}{4}}(1+\tau)^{-\frac{5}{4}}d\tau\\
&\quad +\ C\sqrt{E(t)}\left[\int (1+t-\tau)^{-\frac{5}{2}}(1+\tau)^{-\frac{5}{2}}d\tau\right]^{\frac{1}{2}}
  \left[\int_0^t \|\nabla^2 (u, B)(\tau)\|_{H^1}^2 d\tau\right]^{\frac{1}{2}}\\
&\le C(1+t)^{-\frac{5}{4}}
  +C\delta \sqrt{E(t)}(1+t)^{-\frac{5}{4}}\\
&\le (1+t)^{-\frac{5}{4}}(1+\delta\sqrt{E(t)}),
\end{aligned}
$$
where we have used the fact
$$
\begin{aligned}
&\int_0^t (1+t-\tau)^{-r}(1+\tau)^{-r}d\tau\\
&=\int_0^{\frac{t}{2}}+\int_{\frac{t}{2}}^{t}(1+t-\tau)^{-r}(1+\tau)^{-r}d\tau\\
&\le \left(1+\frac{t}{2}\right)^{-r}\int_0^{\frac{t}{2}}(1+\tau)^{-r}d\tau
+\left(1+\frac{t}{2}\right)^{-r} \int_{\frac{t}{2}}^{t}(1+t-\tau)^{-r}d\tau\\
&\le \left(1+t\right)^{-r},
\end{aligned}
$$
for $r=\frac{5}{2}$ and $r=\frac{5}{4}$ respectively.
Thus, we have the estimate
\begin{equation}\label{266}
\|\nabla(\varrho, u, B)(t)\|_{L^2}^2 \le C (1+t)^{-\frac{5}{2}}(1+\delta  E(t)).
\end{equation}
Inserting \eqref{266} into \eqref{264}, it follows immediately
\begin{equation}\label{267}
\begin{aligned}
\mathcal{E}^2_1(t)
&\le \mathcal{E}^2_1(0)e^{-Ct}+C\int_0^t e^{-C(t-\tau)} (1+\tau)^{-\frac{5}{2}}(1+\delta E(\tau))d\tau\\
&\le \mathcal{E}^2_1(0)e^{-Ct}+C(1+\delta E(t))\int_0^t e^{-C(t-\tau)} (1+\tau)^{-\frac{5}{2}}d\tau\\
&\le \mathcal{E}^2_1(0)e^{-Ct}+C(1+\delta E(t))(1+t)^{-\frac{5}{2}}\\
&\le C(1+\delta E(t))(1+t)^{-\frac{5}{2}},
\end{aligned}
\end{equation}
where we have used the fact
$$
\begin{aligned}
&\int_0^t e^{-C(t-\tau)} (1+\tau)^{-\frac{5}{2}}d\tau\\
&=\int_0^{\frac{t}{2}}+\int_{\frac{t}{2}}^t e^{-C(t-\tau)} (1+\tau)^{-\frac{5}{2}}d\tau\\
&\le e^{-\frac{c}{2}t}\int_0^{\frac{t}{2}}(1+\tau)^{-\frac{5}{2}}d\tau
     +\left(1+\frac{t}{2}\right)^{-\frac{5}{2}} \int_{\frac{t}{2}}^t e^{-C(t-\tau)}d\tau\\
&\le C\left(1+t\right)^{-\frac{5}{2}}.
\end{aligned}
$$
Hence, by virtue of the definition of $E(t)$  and \eqref{267}, it follows immediately
$$
E(t) \le C(1+\delta E(t)),
$$
which, in view of the smallness of $\delta$, gives
$$
E(t) \le C.
$$
Therefore, we have the following time decay rates
\begin{equation}\label{268}
\|\nabla \varrho (t)\|_{H^{1}}^2+\|\nabla u (t)\|_{H^{1}}^2+\|\nabla B (t)\|_{H^{2}}^2 \le C(1+t)^{-\frac{5}{2}}.
\end{equation}
On the other hand, by \eqref{265}, \eqref{non-estimates}, \eqref{268} and Proposition \ref{linearized-Decay},
it is easy to deduce
$$
\begin{aligned}
\|(\varrho, u, B)\|_{L^2}^2
&\le C(1+t)^{-\frac{3}{2}}
  +C\int_0^t \left(\|(S_1, S_2, S_3)\|_{L^1}^2+\|(S_1, S_2, S_3)\|_{L^2}^2\right)(1+t-\tau)^{-\frac{3}{2}}d\tau\\
&\le C(1+t)^{-\frac{3}{2}}
  +C\int_0^t \delta\left(\|\nabla \varrho (\tau)\|_{L^2}^2
  +\|\nabla u(\tau)\|_{H^1}^2
  +\|\nabla B(\tau)\|_{H^1}^2 \right)(1+t-\tau)^{-\frac{3}{2}}d\tau\\
&\le C(1+t)^{-\frac{3}{2}}
  +C\int_0^t (1+t-\tau)^{-\frac{5}{2}}(1+\tau)^{-\frac{3}{2}}d\tau\\
&\le C(1+t)^{-\frac{3}{2}},
\end{aligned}
$$
where we have used the fact
$$
\int_0^t (1+t-\tau)^{-\frac{5}{2}}(1+\tau)^{-\frac{3}{2}}d\tau
\le C(1+t)^{-\frac{3}{2}}.
$$
Hence, we have the following decay rate
\begin{equation}\label{269}
\|\varrho (t)\|_{L^2}^2+\|u (t)\|_{L^2}^2+\|B(t)\|_{L^2}^2
 \le C(1+t)^{-\frac{3}{2}}.
\end{equation}
Therefore, the combination of \eqref{268} and \eqref{269} completes the proof of the lemma.
\end{proof}

Finally, we establish optimal decay rates for the second  order derivatives of magnetic field.
\begin{lemm}\label{lemma2.7}
Under the assumptions of Theorem \ref{THM1}, then the magnetic field has the following decay rate
\begin{equation}\label{271}
\|\nabla^2 B (t)\|_{L^2} \le C(1+t)^{-\frac{7}{4}}.
\end{equation}
\end{lemm}
\begin{proof}
Taking $k=2$ in \eqref{223}, it follows immediately
\begin{equation}\label{272}
\frac{1}{2}\frac{d}{dt}\int |\nabla^2 B|^2 \! dx
+\int |\nabla^3 B|^2dx=\int \nabla^2 S_3 \cdot \nabla^2 B~dx.
\end{equation}
By Holder, Sobolev and Young inequalities, we obtain
\begin{equation}\label{273}
\begin{aligned}
&\int \nabla^2(B\cdot \nabla u+B{\rm div}u)\nabla^2 B~dx\\
&\lesssim (\|\nabla u\|_{L^3}\|\nabla B\|_{L^6}
           +\|B\|_{L^\infty}\|\nabla^2 u\|_{L^2})\|\nabla^3 B\|_{L^2}\\
&\lesssim (\|\nabla u\|_{H^1}\|\nabla^2 B\|_{L^2}
           +\|\nabla B\|_{H^1}\|\nabla^2 u\|_{L^2})\|\nabla^3 B\|_{L^2}\\
&\lesssim \|\nabla (u,B)\|_{H^1}^2\|\nabla^2 (u,B)\|_{L^2}^2+\delta\|\nabla^3 B\|_{L^2}^2.
\end{aligned}
\end{equation}
It follows from \eqref{2311} and \eqref{2313} that
\begin{equation}\label{274}
\int \nabla^2(u\cdot \nabla B)\nabla^2 B~dx
\lesssim \|\nabla u\|_{H^1}^2\|\nabla^2 B\|_{L^2}^2+\delta\|\nabla^3 B\|_{L^2}^2,
\end{equation}
and
\begin{equation}\label{275}
\begin{aligned}
&-\int \nabla^2{\rm curl}\left[g(\varrho)\left(B\cdot \nabla B
                 -\nabla(\frac{1}{2}|B|^2)\right)\right]\nabla^2 B~dx\\
&\lesssim \|\nabla^2 \varrho\|_{L^2}^2\|\nabla B\|_{H^1}^2+\delta\|\nabla^3 B\|_{L^2}^2.\\
\end{aligned}
\end{equation}
Substituting \eqref{273}-\eqref{275} into \eqref{272} and
applying the time decay rates \eqref{261}, then we obtain
\begin{equation}\label{276}
\begin{aligned}
&\frac{d}{dt}\int |\nabla^2 B|^2dx+\int |\nabla^3 B|^2 dx\\
&\lesssim \|\nabla (u,B)\|_{H^1}^2\|\nabla^2 (u,B)\|_{L^2}^2
          +\|\nabla^2 \varrho\|_{L^2}^2\|\nabla B\|_{H^1}^2\\
&\lesssim (1+t)^{-\frac{5}{2}}(1+t)^{-\frac{5}{2}}
           +(1+t)^{-\frac{5}{2}}(1+t)^{-\frac{5}{2}}\\
&\lesssim (1+t)^{-5}.
\end{aligned}
\end{equation}
For some constant $R$ defined below, denoting the time sphere(see \cite{Schonbek})
$$
S_0=\left\{\left. \xi\in \mathbb{R}^3\right| |\xi|\le \left(\frac{R}{1+t}\right)^\frac{1}{2}\right\},
$$
it follows immediately
$$
\begin{aligned}
\int_{\mathbb{R}^3} |\nabla^{3} B|^2 dx
&\ge \int_{\mathbb{R}^3/S_0} |\xi|^6 |\hat{B}|^2d\xi\\
&\ge \frac{R}{1+t}\int_{\mathbb{R}^3/S_0} |\xi|^4 |\hat{B}|^2d\xi\\
&\ge \frac{R}{1+t}\int_{\mathbb{R}^3} |\xi|^4 |\hat{B}|^2 d\xi
     -\left(\frac{R}{1+t}\right)^2\int_{S_0} |\xi|^2|\hat{B}|^2 d\xi,
\end{aligned}
$$
or equivalently
\begin{equation}\label{FSM}
\int_{\mathbb{R}^3} |\nabla^3 B|^2 dx
\ge \frac{R}{1+t}\int_{\mathbb{R}^3} |\nabla^2 B|^2 dx
-\left(\frac{R}{1+t}\right)^2\int_{\mathbb{R}^3} |\nabla B|^2 dx.
\end{equation}
The combination of \eqref{276}, \eqref{FSM} and \eqref{261} yields directly
\begin{equation}\label{277}
\begin{aligned}
&\frac{d}{dt}\int |\nabla^2 B|^2dx+\frac{4}{1+t}\int |\nabla^2 B|^2 dx\\
&\le \frac{16}{(1+t)^2}\int |\nabla B|^2 dx+C(1+t)^{-5}\\
&\lesssim (1+t)^{-2}(1+t)^{-\frac{5}{2}}+(1+t)^{-5}\\
&\le C(1+t)^{-\frac{9}{2}},
\end{aligned}
\end{equation}
where we have chosen $R=4$ in \eqref{FSM}. Multiplying \eqref{277} by $(1+t)^4$, we obtain
\begin{equation}\label{278}
\frac{d}{dt}\left[(1+t)^4 \|\nabla^2 B\|_{L^2}^2\right]\le C(1+t)^{-\frac{1}{2}}.
\end{equation}
Integrating \eqref{278} over $[0,t]$, then we have the following decay rate
$$
\|\nabla^2 B (t)\|_{L^2}^2 \le C(1+t)^{-\frac{7}{2}}.
$$
Therefore, we complete the proof of the lemma.
\end{proof}

\emph{\bf{Proof of Theorem \ref{THM1}:}}
With the help of \eqref{255}, Lemma \ref{lemma2.6} and Lemma \ref{lemma2.7},
we complete the proof of Theorem \ref{THM1}.

\subsection{Proof of Theorem \ref{THM2}}

\quad In this subsection, we establish the decay rates for the time derivatives of
strong solutions.
\begin{lemm}\label{lemma3.1}
Under the assumptions of Theorem \ref{THM1}, the global strong solution $(\varrho, u, B)$
of problem \eqref{eq1}-\eqref{eq4} satisfies
\begin{equation}\label{31}
\begin{aligned}
&\|\varrho_t (t)\|_{H^1}+\|u_t (t)\|_{L^2}\le C(1+t)^{-\frac{5}{4}},\\
&\|B_t (t)\|_{L^2}\le C(1+t)^{-\frac{7}{4}}.
\end{aligned}
\end{equation}
\end{lemm}
\begin{proof}
By virtue of the equation \eqref{eq1}$_1$ and decay rates \eqref{Decay1}, we have
\begin{equation}\label{32}
\begin{aligned}
\|\varrho_t\|_{L^2}
&=\|{\rm div}u+\varrho{\rm div}u+u\cdot \nabla \varrho\|_{L^2}\\
&\le \|{\rm div}u\|_{L^2}+\|\varrho\|_{L^\infty}\|{\rm div}u\|_{L^2}+\|\nabla \varrho\|_{L^3}\|u\|_{L^6}\\
&\le C(1+t)^{-\frac{5}{4}}.
\end{aligned}
\end{equation}
Similarly, it follows immediately
\begin{equation}\label{33}
\begin{aligned}
\|\nabla \varrho_t\|_{L^2}
&=\|\nabla{\rm div}u+\nabla\varrho{\rm div}u+\varrho\nabla{\rm div}u
     +\nabla u\cdot \nabla \varrho+u\cdot \nabla^2 \varrho\|_{L^2}\\
&\lesssim \|\nabla {\rm div}u\|_{L^2}+\|\nabla \varrho\|_{L^3}\|\nabla u\|_{L^6}
     +\|(\varrho, u)\|_{L^\infty}\|\nabla^2 (\varrho, u)\|_{L^2}\\
&\lesssim \|\nabla^2 u\|_{L^2}+\|(\varrho, u)\|_{H^2}\|\nabla^2(\varrho, u)\|_{L^2}\\
&\le C(1+t)^{-\frac{5}{4}}.
\end{aligned}
\end{equation}
By virtue of the equation \eqref{eq1}$_2$, decay rates \eqref{Decay1}
and estimate \eqref{non-estimates}, we have
\begin{equation}\label{34}
\begin{aligned}
\|u_t\|_{L^2}
&=\|\mu\Delta u+(\mu+\nu)\nabla{\rm div}u-\nabla \varrho+S_2\|_{L^2}\\
&\lesssim \|\nabla^2 u\|_{L^2}+\|\nabla \varrho\|_{L^2}+\delta (\|\nabla \varrho\|_{L^2}
           +\|\nabla (u, B)\|_{H^1})\\
&\lesssim (1+t)^{-\frac{5}{4}}+(1+t)^{-\frac{5}{4}}\\
&\le  C(1+t)^{-\frac{5}{4}}.
\end{aligned}
\end{equation}
By virtue of \eqref{eq1}$_3$, \eqref{Decay1}, Holder and Sobolev inequalities, we obtain
\begin{equation}\label{35}
\begin{aligned}
\|B_t\|_{L^2}
&=\left\|\Delta B-u\cdot \nabla B+B\cdot \nabla u-B{\rm div}u
-{\rm curl}\left[g(\varrho)\!\left(B\cdot \nabla B-\frac{1}{2}\nabla(|B|^2)\!\right)\right]\right\|_{L^2}\\
&\lesssim \|\Delta B\|_{L^2}+\|u\|_{L^3}\|\nabla B\|_{L^6}+\|B\|_{L^3}\|\nabla u\|_{L^6}
+\|\nabla g(\varrho)\|_{L^2}\|B\|_{L^6}\|\nabla B\|_{L^6}\\
&\quad +\|g(\varrho)\|_{L^\infty}\|\nabla B\|_{L^3}\|\nabla B\|_{L^6}
       +\|g(\varrho)\|_{L^\infty}\|B\|_{L^\infty}\|\nabla^2 B\|_{L^2}\\
&\lesssim \|\nabla^2 B\|_{L^2}
          +\|(u, B)\|_{H^1}\|\nabla^2 (u, B)\|_{L^2}
          +\|\nabla \varrho\|_{L^2}\|\nabla B\|_{H^1}^2
          +\|\nabla B\|_{H^1}\|\nabla^2 B\|_{L^2}\\
&\lesssim (1+t)^{-\frac{7}{4}}+(1+t)^{-\frac{3}{4}}(1+t)^{-\frac{5}{4}}
          +(1+t)^{-\frac{5}{4}}(1+t)^{-\frac{5}{4}}\\
&\le C(1+t)^{-\frac{7}{4}}.
\end{aligned}
\end{equation}
In view of the decay rates \eqref{32}-\eqref{35}, we complete the proof of the lemma.
\end{proof}

\emph{\bf{Proof of Theorem \ref{THM2}:}}
With the help of Lemma \ref{lemma3.1}, we complete the proof of Theorem \ref{THM2}.

\section{Proof of Theorem \ref{THM3} and Theorem \ref{THM4}}

\quad In this section, we first establish optimal time decay rates for the
higher order spatial derivatives of global classical solutions
under the condition of small initial perturbation in $H^3$-norm and
finite initial perturbation in $L^1$-norm.
Furthermore, we also study the decay rates for the mixed space-time derivatives
of global classical solutions.

First of all, Fan et al.(see $(3.2)$ on Page $430$ in \cite{Fan-Zhou})
have established following estimate
\begin{equation}\label{smallness2}
\|(\varrho, u, B)(t)\|_{H^3}\le C \|(\varrho_0, u_0, B_0)\|_{H^3}\le C\varepsilon_0.
\end{equation}
Thus, the inequality \eqref{2.6} also holds on under the condition of \eqref{smallness}.

\subsection{Proof of Theorem \ref{THM3}}

\quad Just following the idea as Lemma \ref{lemma2.6}, it is easy to establish
optimal decay rates for the global classical solutions.
For the sake of brevity, we only state the results in the following lemma.

\begin{lemm}\label{lemma4.1}
Under the assumptions of Theorem \ref{THM3}, the global classical solution $(\varrho, u, B)$
of problem \eqref{eq1}-\eqref{eq4} satisfies for all $t \ge 0$,
\begin{equation}\label{411}
\|\nabla^k \varrho (t)\|_{H^{3-k}}^2+\|\nabla^k u (t)\|_{H^{3-k}}^2
+\|\nabla^k B (t)\|_{H^{3-k}}^2\le C (1+t)^{-\frac{3}{2}-k},
\end{equation}
where $k=0,1$.
\end{lemm}

Next, we establish optimal time decay rates for the second order
spatial derivatives of magnetic field and enhance the time decay rates
for the third order spatial derivatives of magnetic field.
\begin{lemm}\label{lemma4.2}
Under the assumptions of Theorem \ref{THM3},
then the magnetic field has following decay rate for all $t \ge 0$,
\begin{equation}\label{421}
\|\nabla^2 B(t)\|_{H^1} \le C (1+t)^{-\frac{7}{4}}.
\end{equation}
\end{lemm}
\begin{proof}
Taking $k=3$ in \eqref{223}, it follows immediately
\begin{equation}\label{422}
\begin{aligned}
&\frac{1}{2}\frac{d}{dt}\int |\nabla^3 B|^2 dx+\int \!|\nabla^4 B|^2 dx\\
&=\int \nabla^3\left[-u\cdot \nabla B+B\cdot \nabla u-B{\rm div}u
-{\rm curl}\left(g(\varrho)\!\left(B\cdot \nabla B-\frac{1}{2}\nabla(|B|^2)\right)\right)\right]\nabla^3 B dx\\
&=I\!I\!I_1+I\!I\!I_2+I\!I\!I_3+I\!I\!I_4.
\end{aligned}
\end{equation}
By virtue of \eqref{smallness2}, Holder and Sobolev inequalities, it arrives at
\begin{equation}\label{423}
\begin{aligned}
I\!I\!I_1
&=\int (\nabla^2 u \nabla B+2\nabla u \nabla^2 B+u\nabla^3 B)\nabla^4 B dx\\
&\lesssim (\|\nabla^2 u\|_{L^3}\|\nabla B\|_{L^6}+\|\nabla u\|_{L^3}\|\nabla^2 B\|_{L^6}
         +\|u\|_{L^3}\|\nabla^3 B\|_{L^6})\|\nabla^4 B\|_{L^2}\\
&\lesssim \|\nabla^2 u\|_{L^3}^2\|\nabla^2 B\|_{L^2}^2+\|\nabla u\|_{L^3}^2\|\nabla^3 B\|_{L^2}^2
+(\varepsilon+\varepsilon_0)\|\nabla^4 B\|_{L^2}^2\\
&\lesssim \|\nabla u\|_{H^2}^2\|\nabla^2 B\|_{H^1}^2
+(\varepsilon+\varepsilon_0)\|\nabla^4 B\|_{L^2}^2.
\end{aligned}
\end{equation}
In view of the Sobolev and Young inequalities, we obtain
\begin{equation}\label{424}
\begin{aligned}
I\!I\!I_2
&=-\int (\nabla^2 B\nabla u+2\nabla B \nabla^2 u+B\nabla^3 u)\nabla^4 B dx\\
&\lesssim (\|\nabla^2 B\|_{L^6}\|\nabla u\|_{L^3}+\|\nabla B\|_{L^6}\|\nabla^2 u\|_{L^3}
           +\|B\|_{L^\infty}\|\nabla^3 u\|_{L^2})\|\nabla^4 B\|_{L^2}\\
&\lesssim (\|\nabla^3 B\|_{L^2}\|\nabla u\|_{H^1}+\|\nabla^2 B\|_{L^2}\|\nabla^2 u\|_{H^1}
            +\|\nabla B\|_{H^1}\|\nabla^3 u\|_{L^2})\|\nabla^4 B\|_{L^2}\\
&\lesssim \|\nabla^2 B\|_{H^1}^2\|\nabla u\|_{H^2}^2
           +\|\nabla B\|_{H^1}^2\|\nabla^3 u\|_{L^2}^2
           +\varepsilon\|\nabla^4 B\|_{L^2}^2.\\
\end{aligned}
\end{equation}
In the same manner, we get
\begin{equation}\label{425}
I\!I\!I_3\lesssim
\|\nabla^2 B\|_{H^1}^2\|\nabla u\|_{H^2}^2
           +\|\nabla B\|_{H^1}^2\|\nabla^3 u\|_{L^2}^2
           +\varepsilon\|\nabla^4 B\|_{L^2}^2.
\end{equation}
Applying \eqref{2.6}, \eqref{smallness2}, Holder, Sobolev and Young inequalities,
it is easy to deduce
\begin{equation}\label{426}
\begin{aligned}
I\!I\!I_4
&\lesssim (\|\nabla^3 g(\varrho)\|_{L^2}\|B\|_{L^\infty}\|\nabla B\|_{L^\infty}
          +\|\nabla^2 g(\varrho)\|_{L^3}\|\nabla B\|_{L^6}\|\nabla B\|_{L^\infty}
          )\|\nabla^4 B\|_{L^2}\\
&\quad +(\|\nabla^2 g(\varrho)\|_{L^3}\|B\|_{L^\infty}\|\nabla^2 B\|_{L^6}
         +\|\nabla g(\varrho)\|_{L^6}\|\nabla B\|_{L^6}\|\nabla^2 B\|_{L^6})
         \|\nabla^4 B\|_{L^2}\\
&\quad +(\|\nabla g(\varrho)\|_{L^6}\|B\|_{L^6}\|\nabla^3 B\|_{L^6}
         +\|g(\varrho)\|_{L^\infty}\|\nabla B\|_{L^3}\|\nabla^3 B\|_{L^6})
         \|\nabla^4 B\|_{L^2}\\
&\quad +\|g(\varrho)\|_{L^\infty}\|B\|_{L^\infty}\|\nabla^4 B\|_{L^2}^2\\
&\lesssim \|\nabla B\|_{H^1}^2\|\nabla^2 B\|_{H^1}^2
          +\|\nabla^2 B\|_{L^2}^2\|\nabla^2 B\|_{H^1}^2
          +\|\nabla B\|_{H^1}^2\|\nabla^3 B\|_{L^2}^2
          +(\varepsilon+\varepsilon_0)\|\nabla^4 B\|_{L^2}^2\\
&\lesssim \|\nabla B\|_{H^1}^2\|\nabla^2 B\|_{H^1}^2
          +(\varepsilon+\varepsilon_0)\|\nabla^4 B\|_{L^2}^2.\\
\end{aligned}
\end{equation}
Substituting \eqref{423}-\eqref{426} into \eqref{422}
and applying the smallness of $\varepsilon$ and $\varepsilon_0$,
it is easy to deduce
\begin{equation}\label{427}
\frac{d}{dt}\int |\nabla^3 B|^2 dx+\int |\nabla^4 B|^2 dx
\lesssim \|\nabla(u, B)\|_{H^2}^2\|\nabla^2(u, B)\|_{H^1}^2,
\end{equation}
which, together with the time decay rates \eqref{411}, yields directly
\begin{equation}\label{428}
\frac{d}{dt}\int |\nabla^3 B|^2 dx+\int |\nabla^4 B|^2 dx
\lesssim (1+t)^{-5}.
\end{equation}
Similar to \eqref{FSM}, it is easy to deduce
\begin{equation}\label{429}
\int |\nabla^4 B|^2 dx\ge \frac{5}{1+t}\int |\nabla^3 B|^2 dx-\left(\frac{5}{1+t}\right)^2\int |\nabla^2 B|^2dx.
\end{equation}
The combination of \eqref{411}, \eqref{428} and \eqref{429} gives directly
$$
\begin{aligned}
&\frac{d}{dt}\int |\nabla^3 B|^2 dx+\frac{5}{1+t} \int |\nabla^3 B|^2 dx\\
&\lesssim \frac{25}{(1+t)^2}\int |\nabla^2 B|^2 dx+(1+t)^{-5}\\
&\lesssim (1+t)^{-2}(1+t)^{-\frac{5}{2}}+(1+t)^{-5}\\
&\lesssim (1+t)^{-\frac{9}{2}},
\end{aligned}
$$
which, together with \eqref{277}, yields directly
\begin{equation}\label{4210}
\frac{d}{dt}\int (|\nabla^2 B|^2+|\nabla^3 B|^2) dx+\frac{4}{1+t}\int (|\nabla^2 B|^2+|\nabla^3 B|^2) dx
\lesssim (1+t)^{-\frac{9}{2}}.
\end{equation}
Multiplying \eqref{4210} by $(1+t)^{4}$, it arrives at
$$
\frac{d}{dt}\left[(1+t)^{4}\|\nabla^2 B\|_{H^1}^2\right]\le C(1+t)^{-\frac{1}{2}},
$$
which, integrating over $[0, t]$, gives
$$
\|\nabla^2 B (t)\|_{H^1}^2 \le C(1+t)^{-\frac{7}{2}}.
$$
Therefore, we complete the proof of the lemma.
\end{proof}

In order to establish optimal decay rate for the third order spatial derivatives of
magnetic field, we need to improve the decay rate for the second and third order
spatial derivatives of velocity. Indeed, following the idea as the compressible
MHD equations(see \cite{Gao-Chen-Yao}), it is easy to verify
the time decay rates \eqref{411} also holds on for $k=2$.
For the convenience of readers, we also introduce the method to improve the
decay rates for the second order spatial derivatives of density and velocity here.

\begin{lemm}\label{lemma4.3}
Under the assumptions of Theorem \ref{THM3}, the global classical solution $(\varrho, u, B)$
of Cauchy problem \eqref{eq1}-\eqref{eq4} has
\begin{equation}\label{431}
\frac{d}{dt}\|\nabla^2(\varrho, u)\|_{H^1}^2+\mu \|\nabla^3 u\|_{H^1}^2
\le C_5\left[(1+t)^{-5}+\varepsilon_0 \|\nabla^3 \varrho\|_{L^2}^2\right].
\end{equation}
\end{lemm}
\begin{proof}
Taking $k=2$ specially in \eqref{222}, then we obtain
\begin{equation}\label{432}
\begin{aligned}
&\frac{1}{2}\frac{d}{dt}\int( |\nabla^2 \varrho|^2+\!|\nabla^2 u|^2) dx
+\int( \mu |\nabla^3 u|^2+(\mu+\nu)|\nabla^2 {\rm div} u|^2) dx\\
&=\int \nabla^2 S_1 \cdot \nabla^2 \varrho \ dx+\int \nabla^2 S_2 \cdot \nabla^2 u \ dx.
\end{aligned}
\end{equation}
Integrating by part and applying \eqref{411}, Holder, Sobolev and
Young inequalities, we obtain
\begin{equation}\label{433}
\begin{aligned}
\int \nabla^2 S_1 \cdot \nabla^2 \varrho \ dx
&= \int \nabla (\varrho {\rm div}u+u\cdot\nabla \varrho)\cdot \nabla^3 \varrho dx\\
&\lesssim (\|\nabla \varrho\|_{L^3}\|\nabla u\|_{L^6}
          +\|\nabla^2 u\|_{L^3}\|\varrho\|_{L^6}
          +\|\nabla^2 \varrho\|_{L^3}\| u\|_{L^6})\|\nabla^3 \varrho\|_{L^2}\\
&\lesssim (1+t)^{-5}+\varepsilon\|\nabla^3 \varrho\|_{L^2}^2.
\end{aligned}
\end{equation}
From the estimates \eqref{236} and \eqref{237}, it is obtain the estimates
\begin{equation}\label{434}
\int \nabla^2[-u\!\cdot \!\nabla u-\!h(\varrho)[\mu\Delta u+(\mu+\nu)\nabla {\rm div}u]]
\nabla^2 u dx \lesssim \varepsilon_0 \|\nabla^3 u\|_{L^2}^2.
\end{equation}
Integrating by part and applying \eqref{2.6}, \eqref{411},
Holder and Sobolev inequalities, we obtain
\begin{equation}\label{435}
\begin{aligned}
&\int \nabla^2 [-f(\varrho)\nabla \varrho]\nabla^2 u~ dx\\
&\lesssim (\|f(\varrho)\|_{L^\infty}\|\nabla^2 \varrho\|_{L^2}+\|\nabla f(\varrho)\|_{L^3}\|\nabla \varrho\|_{L^6})
           \|\nabla^3 u\|_{L^2}\\
&\lesssim (\|\varrho\|_{L^\infty}+\|\nabla \varrho\|_{L^3})
           \|\nabla^2 \varrho\|_{L^2}\|\nabla^3 u\|_{L^2}\\
&\lesssim \|\nabla \varrho\|_{H^1}^2\|\nabla^2 \varrho\|_{L^2}^2
           +\varepsilon\|\nabla^3 u\|_{L^2}^2\\
&\lesssim (1+t)^{-5}+\varepsilon\|\nabla^3 u\|_{L^2}^2.
\end{aligned}
\end{equation}
In the same manner, it is easy to deduce
\begin{equation}\label{436}
\begin{aligned}
&\int \nabla^2 \left[g(\varrho)\!\left(B\cdot \nabla B-\frac{1}{2}\nabla(|B|^2)\!\right)\right]\nabla^2 u~dx\\
&\approx\int(\nabla g(\varrho)B \nabla B+g(\varrho)\nabla B \nabla B+g(\varrho)B \nabla^2 B)\nabla^3 u dx\\
&\lesssim (\|\nabla \varrho\|_{L^6}\|B\|_{L^6}\|\nabla B\|_{L^6}
           +\|g(\varrho)\|_{L^\infty}\|\nabla B\|_{L^3}\|\nabla B\|_{L^6})\|\nabla^3 u\|_{L^2}\\
&\quad  +\|g(\varrho)\|_{L^\infty}\|B\|_{L^6}\|\nabla^2 B\|_{L^3}\|\nabla^3 u\|_{L^2}\\
&\lesssim \|\nabla^2 \varrho\|_{L^2}^2\|\nabla B\|_{L^2}^2\|\nabla^2 B\|_{L^2}^2
           +\|\nabla B\|_{L^3}^2\|\nabla^2 B\|_{L^2}^2\\
&\quad     +\|\nabla B\|_{L^2}^2\|\nabla^2 B\|_{L^3}^2
           +\varepsilon\|\nabla^3 u\|_{L^2}^2\\
&\lesssim (1+t)^{-5}+\varepsilon\|\nabla^3 u\|_{L^2}^2.
\end{aligned}
\end{equation}
In view of the estimates $\eqref{434}-\eqref{436}$, it is easy deduce
\begin{equation}\label{437}
\int \nabla^2 S_2 \cdot \nabla^2 u \ dx
\lesssim (1+t)^{-5}+\varepsilon_0 \|\nabla^3 u\|_{L^2}^2.
\end{equation}
Substituting \eqref{433} and \eqref{437} into \eqref{432} and applying the smallness
of $\varepsilon$ and $\varepsilon_0$, we obtain
\begin{equation}\label{438}
\frac{d}{dt}\int (|\nabla^2 \varrho|^2 +|\nabla^2 u|^2) dx
+\mu\int |\nabla^3 u|^2 dx \lesssim (1+t)^{-5}+\varepsilon\|\nabla^3 \varrho\|_{L^2}^2.
\end{equation}
Taking $k=3$ in \eqref{222} specially, then we have
\begin{equation}\label{439}
\begin{aligned}
&\frac{1}{2}\frac{d}{dt}\int (|\nabla^3 \varrho|^2+|\nabla^3 u|^2) dx
+\int (\mu |\nabla^{4} u|^2+(\mu+\nu)|\nabla^3 {\rm div}u|^2 )dx\\
&=\int \nabla^3(-\varrho{\rm div}u-u\cdot \nabla \varrho) \nabla^3 \varrho dx
+\int \nabla^3 \left[-u\!\cdot \!\nabla u-\!h(\varrho)(\mu\Delta u+(\mu+\nu)\nabla {\rm div}u)\right] \nabla^3 u ~dx\\
&\quad +\int \nabla^3 \left[-f(\varrho)\nabla \varrho+g(\varrho)(B\cdot \nabla B-\frac{1}{2}\nabla(|B|^2))\right]
      \nabla^3 u ~dx\\
&=IV_1+IV_2+IV_3+IV_4+IV_5+IV_6+IV_7.
\end{aligned}
\end{equation}
Applying the Holer and Sobolev inequalities, it is easy to deduce
\begin{equation}\label{4310}
\begin{aligned}
IV_1
&\lesssim (\|\nabla^3 \varrho\|_{L^2}\|\nabla u\|_{L^\infty}
           +\|\nabla^2 \varrho\|_{L^6}\|\nabla^2 u\|_{L^3})\|\nabla^3 \varrho\|_{L^2}\\
&\quad \   +(\|\nabla \varrho\|_{L^3}\|\nabla^3 u\|_{L^6}
           +\|\varrho\|_{L^\infty}\|\nabla^4 u\|_{L^2})\|\nabla^3 \varrho\|_{L^2}\\
&\lesssim \varepsilon_0(\|\nabla^3 \varrho\|_{L^2}^2+\|\nabla^4 u\|_{L^2}^2).
\end{aligned}
\end{equation}
Similarly, it is easy to deduce
\begin{equation}\label{4311}
IV_2
\lesssim \varepsilon_0(\|\nabla^3 \varrho\|_{L^2}^2+\|\nabla^4 u\|_{L^2}^2).
\end{equation}
Integrating by part and applying decay rates \eqref{411},
Holder and Sobolev inequalities, it arrives at
\begin{equation}\label{4312}
\begin{aligned}
IV_3
&=\int \nabla^2 (u\cdot \nabla u)\nabla^4 u \ dx\\
&\lesssim (\|\nabla u\|_{L^3}\|\nabla^2 u\|_{L^6}+\|u\|_{L^3}\|\nabla^3 u\|_{L^6})\|\nabla^4 u\|_{L^2}\\
&\lesssim \|\nabla u\|_{L^3}^2\|\nabla^3 u\|_{L^2}^2
          +(\varepsilon+\varepsilon_0)\|\nabla^4 u\|_{L^2}^2\\
&\lesssim (1+t)^{-5}+(\varepsilon+\varepsilon_0)\|\nabla^4 u\|_{L^2}^2.
\end{aligned}
\end{equation}
In view of \eqref{2.6}, \eqref{411},  Holer and Sobolev inequalities, we obtain
\begin{equation}\label{4313}
\begin{aligned}
IV_4
&\approx \int \nabla^2 (h(\varrho) \nabla^2 u)\nabla^4 u \ dx\\
&=\int (\nabla^2 h(\varrho)\nabla^2 u+2\nabla h(\varrho)\nabla^3 u+h(\varrho)\nabla^4 u)\nabla^4 u dx\\
&\lesssim (\|\nabla \varrho\|_{L^6}^2\|\nabla^2 u\|_{L^6}
           +\|\nabla^2 \varrho\|_{L^3}\|\nabla^2 u\|_{L^6})\|\nabla^4 u\|_{L^2}\\
&\quad   +(\|\nabla \varrho\|_{L^3}\|\nabla^3 u\|_{L^6}
           +\|h(\varrho)\|_{L^\infty}\|\nabla^4 u\|_{L^2})\|\nabla^4 u\|_{L^2}\\
&\lesssim \|\nabla^2 \varrho\|_{H^1}^2\|\nabla^3 u\|_{L^2}^2
          +(\varepsilon+\varepsilon_0)\|\nabla^4 u\|_{L^2}^2\\
&\lesssim (1+t)^{-5}+(\varepsilon+\varepsilon_0)\|\nabla^4 u\|_{L^2}^2.
\end{aligned}
\end{equation}
Similarly, it is easy to deduce immediately
\begin{equation}\label{4314}
\begin{aligned}
IV_5
&=\int (f(\varrho)\nabla^3 \varrho+2\nabla f(\varrho)\nabla^2 \varrho
       +\nabla^2 f(\varrho)\nabla \varrho)\nabla^4 udx\\
&\lesssim (\|f(\varrho)\|_{L^\infty}\|\nabla^3 \varrho\|_{L^2}
          +\|\nabla \varrho\|_{L^6}^2\|\nabla \varrho\|_{L^6}
          +\|\nabla^2 \varrho\|_{L^3}\|\nabla \varrho\|_{L^6})\|\nabla^4 u\|_{L^2}\\
&\lesssim \|\nabla \varrho\|_{H^1}^2\|\nabla^3 \varrho\|_{L^2}^2
          +\|\nabla^2 \varrho\|_{L^2}^6
          +\|\nabla^2 \varrho\|_{L^3}^2\|\nabla^2 \varrho\|_{L^2}^2
          +\varepsilon\|\nabla^4 u\|_{L^2}^2\\
&\lesssim (1+t)^{-5}+\varepsilon\|\nabla^4 u\|_{L^2}^2.
\end{aligned}
\end{equation}
Integrating by part and applying \eqref{2.6}, \eqref{411},
Holder and Sobolev inequalities, we get
\begin{equation}\label{4315}
\begin{aligned}
IV_6
&=-\int (\nabla^2 g(\varrho)B\nabla B+2\nabla g(\varrho)\nabla(B\nabla B)+g(\varrho)\nabla^2(B\nabla B))\nabla^4 u dx\\
&\lesssim (\|\nabla^2 g(\varrho)\|_{L^6}\|B\|_{L^6}\|\nabla B\|_{L^6}
     +\|\nabla g(\varrho)\|_{L^6}\|\nabla B\|_{L^6}^2)\|\nabla^4 u\|_{L^2}\\
&\quad    \ +(\|\nabla g(\varrho)\|_{L^6}\|B\|_{L^6}\|\nabla^2 B\|_{L^6}
            +\|g(\varrho)\|_{L^\infty}\|\nabla B\|_{L^6}\|\nabla^2 B\|_{L^3})\|\nabla^4 u\|_{L^2}\\
&\quad   \  +\|g(\varrho)\|_{L^\infty}\|B\|_{L^\infty}\|\nabla^3 B\|_{L^2}\|\nabla^4 u\|_{L^2}\\
&\lesssim \|\nabla B\|_{L^2}^2\|\nabla^2 B\|_{L^2}^2
     +\|\nabla^2 B\|_{L^2}^2\|\nabla^2 B\|_{H^1}^2
     +\varepsilon \|\nabla^4 u\|_{L^2}^2\\
&\lesssim \|\nabla B\|_{H^1}^2\|\nabla^2 B\|_{H^1}^2+\varepsilon \|\nabla^4 u\|_{L^2}^2\\
&\lesssim (1+t)^{-5}+\varepsilon \|\nabla^4 u\|_{L^2}^2.
\end{aligned}
\end{equation}
In the same manner, it arrives at directly
\begin{equation}\label{4316}
IV_7 \lesssim (1+t)^{-5}+\varepsilon \|\nabla^4 u\|_{L^2}^2.
\end{equation}
Substituting \eqref{4310}-\eqref{4316} into \eqref{439}, we obtain
$$
\frac{d}{dt}\int (|\nabla^3 \varrho|^2+|\nabla^3 u|^2) dx
+\mu\int |\nabla^4 u|^2 dx
\lesssim (1+t)^{-5}+\varepsilon_0\|\nabla^3 \varrho\|_{L^2}^2,
$$
which, together with \eqref{438}, completes the proof of the lemma.
\end{proof}

Next, we establish the inequality to recover the dissipation estimate for $\varrho$.

\begin{lemm}\label{lemma4.4}
Under the assumptions in Theorem \ref{THM3}, the global classical solution $(\varrho, u, B)$
of Cauchy problem \eqref{eq1}-\eqref{eq4} satisfies
\begin{equation}\label{441}
\frac{d}{dt}\int \nabla^2 u\cdot \nabla^3 \varrho dx+C_6 \int |\nabla^3 \varrho|^2 dx
\le C_7\left[ (1+t)^{-5}+\|\nabla^3 u\|_{H^1}^2\right].
\end{equation}
\end{lemm}
\begin{proof}
Taking $\nabla^2$ operator on both hand sides of $\eqref{eq1}_1$,
multiplying by $\nabla^2 \varrho$ and integrating over $\mathbb{R}^3$, then we have
\begin{equation}\label{442}
\int (\nabla^2 u_t \cdot \nabla^3 \varrho+|\nabla^3 \varrho|^2) dx
=\int [\mu \Delta \nabla^2 u+(\mu+\nu)\nabla^3{\rm div}u+\nabla^2 S_2]\nabla^3 \varrho~ dx.
\end{equation}
In order to deal with the term $\int \nabla^2 u_t \cdot \nabla^3 \varrho dx$, we turn the time derivatives
of velocity to density and apply the transport equation $\eqref{eq1}_1$. More precisely, we get
\begin{equation}\label{443}
\begin{aligned}
\int \nabla^2 u_t \cdot \nabla^3 \varrho dx
&=\frac{d}{dt}\int \nabla^2 u \cdot \nabla^3 \varrho dx-\int \nabla^2 u \cdot\nabla^3 \varrho_t dx\\
&=\frac{d}{dt}\int \nabla^2 u \cdot \nabla^3 \varrho dx+\int \nabla^2 {\rm div} u
                                                        \cdot \nabla^2 \varrho_t dx\\
&=\frac{d}{dt}\int \nabla^2 u \cdot \nabla^3 \varrho dx
   -\int \nabla^2 {\rm div} u \cdot \nabla^2({\rm div}u+\varrho {\rm div}u+u\nabla \varrho)dx.
\end{aligned}
\end{equation}
Substituting \eqref{443} into \eqref{442}, it arrives at
\begin{equation}\label{444}
\begin{aligned}
&\frac{d}{dt}\int \nabla^2 u \cdot \nabla^3 \varrho dx+\int |\nabla^3 \varrho|^2 dx\\
&=\int |\nabla^2 {\rm div} u|^2 dx+\int \nabla^2 {\rm div} u \cdot \nabla^2(\varrho {\rm div}u+u\nabla \varrho)dx
+\int \nabla^2 S_2 \cdot \nabla^3 \varrho dx\\
&\quad +\int [\mu \Delta \nabla^2 u+(\mu+\nu)\nabla^3{\rm div}u]\nabla^3 \varrho dx.
\end{aligned}
\end{equation}
With the help of time decay rates \eqref{411}, Holder and Sobolev inequalities, we obtain
\begin{equation}\label{445}
\begin{aligned}
&\int \nabla^2 {\rm div} u \cdot \nabla^2(\varrho {\rm div}u+u\nabla \varrho)dx\\
&=-\int \nabla^3{\rm div}u \cdot \nabla(\varrho {\rm div}u+u\nabla \varrho)dx\\
&\lesssim (\|\nabla \varrho\|_{L^3}\|\nabla u\|_{L^6}
         +\|\varrho\|_{L^6}\|\nabla^2 u\|_{L^3}+\|u\|_{L^6}\|\nabla^2 \varrho\|_{L^3})
         \|\nabla^4 u\|_{L^2}\\
&\lesssim \|\nabla \varrho\|_{H^1}^2\|\nabla^2 u\|_{H^1}^2
         +\|\nabla u\|_{L^2}^2\|\nabla^2 \varrho\|_{H^1}^2
         +\varepsilon \|\nabla^4 u\|_{L^2}^2\\
&\lesssim (1+t)^{-5}+\|\nabla^4 u\|_{L^2}^2.\\
\end{aligned}
\end{equation}
On the other hand, just following the idea as \eqref{4312}-\eqref{4315}, we have
\begin{equation}\label{446}
\int \nabla^2 S_2 \cdot \nabla^3 \varrho dx
\lesssim (1+t)^{-5}+\|\nabla^4 u\|_{L^2}^2+\varepsilon\|\nabla^3 \varrho\|_{L^2}^2
\end{equation}
and
\begin{equation}\label{447}
\int [\mu \Delta \nabla^2 u+(\mu+\nu)\nabla^3{\rm div}u]\nabla^3 \varrho dx
\lesssim \|\nabla^4 u\|_{L^2}^2+\varepsilon\|\nabla^3 \varrho\|_{L^2}^2.
\end{equation}
Plugging \eqref{445}-\eqref{447} into \eqref{444}, we complete the proof of lemma.
\end{proof}

Now, we establish optimal decay rates for the second order spatial
derivatives of density and velocity.

\begin{lemm}\label{lemma4.5}
Under the assumptions in Theorem \ref{THM3},
then the density and velocity have following decay rate
\begin{equation}\label{451}
\|\nabla^2 \varrho(t)\|_{H^1}+\|\nabla^2 u(t)\|_{H^1}
\le  C(1+t)^{-\frac{7}{4}}
\end{equation}
for all $t\ge T^* ( T^* \ \text{is a constant defined below})$.
\end{lemm}
\begin{proof}
Multiplying \eqref{441} by $\frac{2 C_5 \varepsilon_0}{C_6} $ and adding to \eqref{431}, then we have
\begin{equation}\label{452}
\frac{d}{dt} \mathcal{E}_2^3(t)+ C_8\int (|\nabla^3 \varrho|^2+|\nabla^3 u|^2+|\nabla^4 u|^2 )dx
\le C_9(1+t)^{-5},
\end{equation}
where $\mathcal{E}_2^3(t)$ is defined as
\begin{equation}\label{453}
\mathcal{E}_2^3(t)=\|\nabla^2 \varrho \|_{H^1}^2+\|\nabla^2 u\|_{H^1}^2
+\frac{2 C_5 \varepsilon_0}{C_6} \int \nabla^2 u\cdot \nabla^3 \varrho dx.
\end{equation}
By virtue of the smallness of $\varepsilon_0$, we have
\begin{equation}\label{454}
C_{10}^{-1}\|\nabla^2(\varrho,u)\|_{H^1}^2
\le
\mathcal{E}_2^3(t)
\le
C_{10} \|\nabla^2(\varrho,u)\|_{H^1}^2.
\end{equation}
It follows directly from \eqref{452} that
\begin{equation}\label{455}
\frac{d}{dt} \mathcal{E}_2^3(t)+
\frac{C_8}{2}\int (|\nabla^3 \varrho|^2+|\nabla^3 \varrho|^2+|\nabla^3 u|^2+|\nabla^4 u|^2) dx
\le C_9 (1+t)^{-5}.
\end{equation}
In the same manner as \eqref{FSM}, we have
\begin{equation}\label{456}
\int |\nabla^3 \varrho|^2 dx\ge \frac{R}{1+t}\int |\nabla^2 \varrho|^2 dx
                            -\left(\frac{R}{1+t}\right)^2\int |\nabla \varrho|^2 dx,
\end{equation}
and
\begin{equation}\label{457}
\|\nabla^3 u\|_{H^1}^2
\ge \frac{R}{1+t}\|\nabla^2 u\|_{H^1}^2
     -\left(\frac{R}{1+t}\right)^2\|\nabla u\|_{H^1}^2.
\end{equation}
Plugging \eqref{456} and \eqref{457} into \eqref{455}, it follows
\begin{equation}\label{458}
\begin{aligned}
&\frac{d}{dt} \mathcal{E}_2^3(t)
+\frac{C_8}{2}\left[\frac{R}{1+t}\int (|\nabla^2 \varrho|^2+|\nabla^2 u|^2+|\nabla^3 u|^2)dx
                    +\int |\nabla^3 \varrho|^2 dx\right]\\
&\lesssim \left(\frac{R}{1+t}\right)^2\int(|\nabla \varrho|^2+|\nabla u|^2+|\nabla^2 u|^2)dx
          +(1+t)^{-5}\\
&\lesssim (1+t)^{-2} (1+t)^{-\frac{5}{2}}+(1+t)^{-5}\\
&\lesssim (1+t)^{-\frac{9}{2}}.
\end{aligned}
\end{equation}
For some large time $t\ge R-1$, we have
$$
\frac{R}{1+t}\le 1,
$$
which implies
\begin{equation}\label{459}
\frac{R}{1+t}\int |\nabla^3 \varrho|^2 dx\le \int |\nabla^3 \varrho|^2 dx.
\end{equation}
Combining \eqref{458} with \eqref{459}, it is easy to deduce
$$
\frac{d}{dt}\mathcal{E}_2^3(t) +\frac{C_8 R}{2(1+t)}\|\nabla^2 (\varrho, u)\|_{H^1}^2
\lesssim (1+t)^{-\frac{9}{2}},
$$
which, together with the equivalent relation \eqref{454}, yields
\begin{equation}\label{4510}
\frac{d}{dt} \mathcal{E}_2^3(t)+\frac{C_8 R}{2C_{10} (1+t)}\mathcal{E}_2^3(t)
\lesssim (1+t)^{-\frac{9}{2}}.
\end{equation}
If choosing
$
R=\frac{8C_{10}}{C_8}
$
in \eqref{4510}, it is easy to deduce
\begin{equation}\label{4511}
\frac{d}{dt} \mathcal{E}_2^3(t)+\frac{4}{1+t}\mathcal{E}_2^3(t)\lesssim (1+t)^{-\frac{9}{2}},
\end{equation}
for all $t \ge T^*:=\frac{8C_{10}}{C_8}-1$.
Multiplying \eqref{4511} by $(1+t)^{4}$, then we have
\begin{equation}\label{4512}
\frac{d}{dt} \left[(1+t)^{4} \mathcal{E}_2^3(t)\right] \lesssim (1+t)^{-\frac{1}{2}}.
\end{equation}
Integrating \eqref{4512} over $[0, t]$, then we get
$$
\mathcal{E}_2^3(t) \lesssim  (1+t)^{-\frac{7}{2}},
$$
which, together with the equivalent relation \eqref{454}, gives
$$
\|\nabla^2 \varrho (t)\|_{H^1}^2+\|\nabla^2 u (t)\|_{H^1}^2  \le  C(1+t)^{-\frac{7}{2}}.
$$
Therefore, we complete the proof of the lemma.
\end{proof}

Finally, we establish optimal decay rate for the third order spatial derivatives
of magnetic field.

\begin{lemm}\label{lemma4.6}
Under the assumption of Theorem \ref{THM3},
then the magnetic field has following decay rate for all $t\ge T^*$,
\begin{equation}
\|\nabla^3 B (t)\|_{L^2} \le C(1+t)^{-\frac{9}{4}}.
\end{equation}
\end{lemm}
\begin{proof}
By virtue of \eqref{427} and applying time decay rates \eqref{421} and \eqref{451}, we obtain
$$
\begin{aligned}
&\frac{d}{dt}\int |\nabla^3 B|^2 dx+\int |\nabla^4 B|^2 dx\\
&\lesssim\|\nabla(u, B)\|_{H^2}^2\|\nabla^2(u, B)\|_{H^1}^2\\
&\lesssim (1+t)^{-\frac{5}{2}}(1+t)^{-\frac{7}{2}}\\
&\lesssim (1+t)^{-6},
\end{aligned}
$$
which, together with \eqref{429}, gives  directly
\begin{equation}\label{462}
\begin{aligned}
&\frac{d}{dt}\int |\nabla^3 B|^2 dx+\frac{5}{1+t}\int |\nabla^3 B|^2 dx\\
&\lesssim (1+t)^{-2}\|\nabla^2 B\|_{L^2}^2+(1+t)^{-6}\\
&\lesssim (1+t)^{-2}(1+t)^{-\frac{7}{2}}+(1+t)^{-6}\\
&\lesssim (1+t)^{-\frac{11}{2}}.
\end{aligned}
\end{equation}
Multiplying \eqref{462} by $(1+t)^{5}$ and integrating the resulting inequality
over $[0, t]$, we obtain
$$
\|\nabla^3 B (t)\|_{L^2}^2 \lesssim (1+t)^{-\frac{9}{2}}.
$$
Therefore, we complete the proof of the lemma.
\end{proof}

\emph{\bf{Proof of Theorem \ref{THM3}:}} With the help of Lemma \ref{lemma4.1},
Lemma \ref{lemma4.2}, Lemma \ref{lemma4.5} and Lemma \ref{lemma4.6},
we complete the proof of Theorem \ref{THM3}.

\subsection{Proof of Theorem \ref{THM4}}

\quad In this section, we establish the time decay rates for the mixed space-time derivatives of global classical solutions.

\begin{lemm}\label{lemma5.1}
Under the assumptions in Theorem \ref{THM3}, the global classical solution $(\varrho, u, B)$
of Cauchy problem \eqref{eq1}-\eqref{eq4} has the time decay rates
$$
\begin{aligned}
&\|\nabla^k \varrho_t(t)\|_{H^{2-k}}+\|\nabla^k u_t(t)\|_{L^2}
  \le C(1+t)^{-\frac{5+2k}{4}},\\
&\|\nabla^k B_t(t)\|_{L^2} \le C(1+t)^{-\frac{7+2k}{4}},\\
\end{aligned}
$$
where $k=0,1$.
\end{lemm}
\begin{proof}
First of all, applying the estimate \eqref{33} and time decay rates \eqref{Decay3}, we obtain
\begin{equation}\label{511}
\|\nabla \varrho_t\|_{L^2}^2 \lesssim (1+t)^{-\frac{7}{2}}.
\end{equation}
Applying the equation \eqref{eq1}$_1$, time decay rates \eqref{Decay3},
Holder and Sobolev inequalities, it arrives at
\begin{equation}\label{512}
\begin{aligned}
\|\nabla^2 \varrho_t\|_{L^2}^2
&=\|-\nabla^2{\rm div}u- \nabla^2(\varrho {\rm div}u+u \cdot \nabla \varrho)\|_{L^2}^2\\
&\lesssim \|\nabla^3 u\|_{L^2}^2
          +\|\nabla^2 (\varrho,u)\|_{L^3}^2\|\nabla(\varrho,u)\|_{L^6}^2\\
&\quad \  +\|(\varrho,u)\|_{L^\infty}^2\|\nabla^3 (\varrho,u)\|_{L^2}^2\\
&\lesssim (1+t)^{-\frac{7}{2}}.
\end{aligned}
\end{equation}
Combining \eqref{511}-\eqref{512} with \eqref{32}, it is easy to obtain
\begin{equation}\label{513}
\|\nabla^k \varrho_t(t)\|_{H^{2-k}}^2 \le (1+t)^{-\frac{5+2k}{2}},
\end{equation}
where $k=0,1$.
Secondly, in view of the equation \eqref{eq1}$_2$, \eqref{non-estimates} and Holder inequality,
we get
$$
\begin{aligned}
\|\nabla u_t\|_{L^2}^2
&=\|\mu \Delta \nabla u+(\mu+\nu)\nabla^2 {\rm div} u-\nabla^2 \varrho+\nabla S_2\|_{L^2}^2\\
&\lesssim \|\nabla^3 u\|_{L^2}^2+\|\nabla^2 \varrho\|_{L^2}^2
          +\|\nabla (\varrho, B)\|_{H^1}^2\|\nabla^2 (u, B)\|_{H^1}^2\\
&\quad   +\delta (\|\nabla^2 \varrho\|_{L^2}^2
                  +\|\nabla^2 u\|_{L^2}^2
                  +\|\nabla^2 B\|_{L^2}^2)\\
&\lesssim (1+t)^{-\frac{7}{2}},
\end{aligned}
$$
which, together with \eqref{34}, gives directly
\begin{equation}\label{514}
\|\nabla^k u_t(t)\|_{L^2}^2 \le C(1+t)^{-\frac{5+2k}{2}},
\end{equation}
where $k=0,1$. Finally, it follows from $\eqref{eq1}_3$,
\eqref{273}-\eqref{275}, Holder and Sobolev inequalities that
$$
\begin{aligned}
\|\nabla B_t\|_{L^2}^2
&=\|\nabla \Delta B+\nabla S_3\|_{L^2}^2\\
&\lesssim \|\nabla^3 B\|_{L^2}^2+\|\nabla (u, B)\|_{H^1}^2\|\nabla^2 (u, B)\|_{L^2}^2
          +\|\nabla^2 \varrho\|_{L^2}^2\|\nabla B\|_{H^1}^2\\
&\lesssim (1+t)^{-\frac{9}{2}}+(1+t)^{-\frac{5}{2}}(1+t)^{-\frac{7}{2}}
          +(1+t)^{-\frac{7}{2}}(1+t)^{-\frac{5}{2}}\\
&\lesssim (1+t)^{-\frac{9}{2}},
\end{aligned}
$$
which, together with \eqref{35}, gives directly
\begin{equation}\label{515}
\|\nabla^k B_t(t)\|_{L^2}^2 \le C(1+t)^{-\frac{7+2k}{2}},
\end{equation}
where $k=0,1$.
Combining \eqref{513}, \eqref{514} with \eqref{515}, then we complete the proof of lemma.
\end{proof}

\emph{\bf{Proof of Theorem \ref{THM4}:}} With the help of Lemma \ref{lemma5.1},
we complete the proof of Theorem \ref{THM4}.

\section*{Acknowledgements}

This research was supported in part by NNSFC(Grant No.11271381) and China 973 Program(Grant No. 2011CB808002).

\phantomsection


\addcontentsline{toc}{section}{\refname}
\begin{thebibliography}{99}

\bibitem{Forbes}
T. G. Forbes,
Magnetic reconnection in solar flares,
Geophys. Astrophys. Fluid Dyn. 62 (1991) 15-36.

\bibitem{Homann}
H. Homann, R. Grauer,
Bifurcation analysis of magnetic reconnection in Hall-MHD systems,
Phys. D 208 (2005) 59-72.

\bibitem{Wardle}
M. Wardle,
Star formation and the Hall effect,
Astrophys. Space Sci. 292 (2004) 317-323.

\bibitem{Balbus-Terquem}
S. A. Balbus, C. Terquem,
Linear analysis of the Hall effect in protostellar disks,
Astrophys. J. 552 (2001) 235-247.

\bibitem{Shalybkov-Urpin}
D. A. Shalybkov, V. A. Urpin,
The Hall effect and the decay of magnetic fields,
Astron. Astrophys. (1997) 685-690.

\bibitem{Mininni}
P. D. Mininni, D. O. G\`{o}mez, S. M. Mahajan,
Dynamo action in magnetohydrodynamics and Hall magnetohydrodynamics,
Astrophys. J. 587 (2003) 472-481.

\bibitem{Liu}
M. Acheritogaray, P. Degond, A. Frouvelle, J. G. Liu,
Kinetic formulation and global existence for the Hall-Magneto-hydrodynamics system,
Kinet. Relat. Models  4  (2011)  901-918.


\bibitem{Chae-Degond-Liu}
D. Chae, P. Degond, J. G. Liu,
Well-posedness for Hall-magnetohydrodynamics,
Ann. Inst. H. Poincar\'{e} Anal. Non Lin\'{e}aire  31  (2014) 555-565.

\bibitem{Chae-Lee}
D. Chae, J. Lee,
On the blow-up criterion and small data global existence for the Hall-magnetohydrodynamics,
J. Differential Equations  256  (2014)  3835-3858.

\bibitem{Fan-Li-Nakamura}
J. S. Fan, F. C. Li, G. Nakamura,
Regularity criteria for the incompressible Hall-magnetohydrodynamic equations,
Nonlinear Anal.  109  (2014) 173-179.

\bibitem{Fan-Ozawa}
J. S. Fan, T. Ozawa,
Regularity criteria for the density-dependent Hall-magnetohydrodynamics,
Appl. Math. Lett.  36  (2014) 14-18.


\bibitem{Maicon-Lucas}
Maicon J. Benvenutti, Lucas C. F. Ferreira,
Existence and stability of global large strong solutions for the Hall-MHD system,
arXiv:1412.8516.

\bibitem{Fan-Huang-Nakamura}
J. S. Fan, S. X. Huang, G. Nakamura,
Well-posedness for the axisymmetric incompressible viscous Hall-magnetohydrodynamic equations,
Appl. Math. Lett.  26  (2013)  963-967.


\bibitem{Chae-Schonbek}
D. Chae, M. E. Schonbek,
On the temporal decay for the Hall-magnetohydrodynamic equations,
J. Differential Equations  255  (2013) 3971-3982.

\bibitem{Weng-Shangkun}
S. K. Weng,
On analyticity and temporal decay rates of solutions to the viscous resistive Hall-MHD system,
arXiv:1412.8239.

\bibitem{Fan-Zhou}
J. S. Fan, A. Alsaedi, T. Hayat, G. Nakamura, Y. Zhou,
On strong solutions to the compressible Hall-magnetohydrodynamic system,
Nonlinear Anal. Real World Appl.  22  (2015), 423-434.

\bibitem{Li-Yu}
F. C. Li, H. J. Yu,
Optimal decay rate of classical solutions to the compressible magnetohydrodynamic equations,
Proc. Roy. Soc. Edinburgh Sect. A 141 (2011) 109-126.

\bibitem{Chen-Tan}
Q. Chen, Z. Tan,
Global existence and convergence rates of smooth solutions for the compressible
magnetohydrodynamic equations,
Nonlinear Anal. 72 (2010) 4438-4451.


\bibitem {Guo-Wang}
Y. Guo, Y. J. Wang,
Decay of dissipative equations and negative Sobolev spaces,
Comm. Partial Differential Equations  37  (2012)  2165-2208.

\bibitem{Tan-Wang}
Z. Tan, H. Q.Wang,
Optimal decay rates of the compressible magnetohydrodynamic equations,
Nonlinear Anal. Real World Appl.  14  (2013)  188-201.


\bibitem{Schonbek}
M. E. Schonbek,
$L^2$ decay for weak solutions of the Navier-Stokes equations,
Arch. Rational Mech. Anal.  88  (1985) 209-222.


\bibitem{Gao-Tao-Yao}
J. C. Gao, Q. Tao, Z. A. Yao,
Long-time Behavior of Solution for the Compressible Nematic Liquid Crystal Flows in $\mathbb{R}^3$,
arXiv:1503.02865.

\bibitem{Gao-Chen-Yao}
J. C. Gao, Y. H. Chen, Z. A. Yao.
Long-time Behavior of Solution to the Compressible Magnetohydrodynamic Equations,
Preprint.

\bibitem {Tan-Wang2}
Y. J. Wang, Z. Tan,
Global existence and optimal decay rate for the strong solutions in $H^2$ to the
compressible Navier-Stokes equations,
Appl. Math. Lett. 24 (2011) 1778-1784.


\bibitem {Hu-Wu2}
X. P. Hu, G. C. Wu,
Global existence and optimal decay rates for three-dimensional compressible viscoelastic flows,
SIAM J. Math. Anal.  45  (2013) 2815-2833.

\bibitem{Wang-Wang}
W. J. Wang, W. K. Wang,
Decay rates of the compressible Navier-Stokes-Korteweg equations with potential forces,
Discrete Contin. Dyn. Syst. 35 (2015) 513-536.

\bibitem{Wang-Wen-Jun}
W. J. Wang,
Large time behavior of solutions to the compressible Navier-Stokes equations with potential force,
J. Math. Anal. Appl. 423 (2015) 1448-1468.


\bibitem{Nirenberg}
L. Nirenberg,
On elliptic partial differential euations,
Ann.Scuola Norm. Sup. Pisa 13 (1959) 115-162.

\bibitem{Matsumura-Nishida}
A. Matsumura, T. Nishida,
The initial value problems for the equations of motion of viscous and heat-conductive gases,
J.Math.Kyoto Univ. 20 (1980) 67-104.

\bibitem{Duan-Yang}
R. J. Duan, H. X. Liu, S. J. Ukai, T. Yang,
Optimal $L^p-L^q$ convergence rates for the compressible
Navier-Stokes equations with potential force,
J.Differential Equations 238 (2007) 220-233.

\end{thebibliography}
\end{document}